\documentclass[11pt,reqno]{amsart}
\usepackage{amsthm,amssymb,amsfonts,amsmath}
\usepackage{amsrefs}
\usepackage{fullpage}
\usepackage{bm, bbm, mathrsfs}
\usepackage{graphics, graphicx}
\usepackage[mathscr]{euscript}
\usepackage{stmaryrd}
\usepackage{caption, subfig}

\theoremstyle{plain}
\newtheorem{theorem}{Theorem}
\newtheorem{prop}[theorem]{Proposition}

\newtheorem{lemma}[theorem]{Lemma}

\theoremstyle{definition}

\theoremstyle{remark}
\newtheorem{remark}{Remark}
\newtheorem{todo}{TODO}

\newcommand{\me}{\ensuremath{\mathrm{e}}}

\newcommand{\dif}{\ensuremath{\mathrm{d}}}

\newcommand{\RR}{\ensuremath{\mathbb{R}}}
\newcommand{\lam}{\ensuremath{\lambda}}
\newcommand{\Ea}{\ensuremath{E_\mathrm{A}}}
\newcommand{\ui}{\ensuremath{u_\mathrm{ig}}}
\newcommand{\tr}{\ensuremath{\mathrm{t}}}

\DeclareMathOperator{\re}{Re}
\DeclareMathOperator{\im}{Im}

\newcommand{\norm}[2]{\ensuremath{\|#1\|_{#2}}}

\newcommand{\intr}{\int_\RR}

\newcommand{\spm}{{\ensuremath{{\scriptscriptstyle\pm}}}}
\renewcommand{\sp}{{\ensuremath{{\scriptscriptstyle +}}}}
\newcommand{\sm}{{\ensuremath{{\scriptscriptstyle -}}}}

\newcommand{\beq}{\begin{equation}}
\newcommand{\eeq}{\end{equation}}
\newcommand\br{\begin{remark}}
\newcommand\er{\end{remark}}
\newcommand\bt{\begin{todo}}
\newcommand\et{\end{todo}}

\numberwithin{equation}{section}

\title{Stability of viscous weak detonation waves for Majda's model}
\author{Jeffrey Hendricks}
\address{Department of Mathematics, Brigham Young University, Provo, UT 84603}
\email{jjhendricks@math.byu.edu}
\author{Jeffrey Humpherys}
\address{Department of Mathematics, Brigham Young University, Provo, UT 84602}
\email{jeffh@math.byu.edu}
\thanks{Humpherys was partially supported by NSF grant DMS-0847074 (CAREER)}
\author{Gregory Lyng}
\address{Department of Mathematics, University of Wyoming, Laramie, WY 82071}
\email{glyng@uwyo.edu}
\thanks{Lyng was partially supported by NSF grant DMS-0845127 (CAREER)}
\author{Kevin Zumbrun}
\address{Department of Mathematics, Indiana University, Bloomington, IN 47405}
\email{kzumbrun@indiana.edu}
\thanks{Zumbrun was partially supported by NSF grant DMS-0801745}
\date{Last Updated:  \today}

\date{Last Updated:  \today}
\begin{document}
\begin{abstract} Continuing the program initiated by Humpherys, Lyng, \& Zumbrun \cite{HLZ_majda} for strong detonation waves, we use a combination of  analytical and numerical Evans-function techniques to analyze the spectral stability of \emph{weak detonation waves} in a simplified model for gas-dynamical combustion. Combining these new spectral stability results with the pointwise Green function analysis of Lyng, Raoofi, Texier, \& Zumbrun \cite{LRTZ_JDE07}, we conclude that these waves are nonlinearly stable.
The principal novelty of this analysis is the treatment of weak detonation waves. In contrast to the case of strong detonation waves, weak detonation waves are undercompressive and the stability of these waves is delicate and has not been treated by standard weighted-norm techniques. The present analysis thus provides a case study illustrating the flexibility and power of the Evans-function-based approach to stability. As in the case of strong detonations, we find that all tested waves are spectrally stable, hence nonlinearly stable.
\end{abstract}
\maketitle

\section{Introduction}\label{sec:intro}
\subsection{Weak detonation waves and stability}\label{ssec:back}

\subsubsection{Weak detonations}
In the classical theory of combustion \cite{CourantFriedrichs}, detonation waves are classified as one of three types: strong, weak, and Chapman--Jouguet (CJ) . All of these are compressive waves---the pressure and density increase following the wave. We recall that the CJ detonation is distinguished in the theory; in particular, the CJ detonation travels at the slowest speed of all detonations. Moreover, the point representing the burned state of the CJ detonation separates the detonation branch of the Hugoniot curve into two pieces. The possible burned states on the lower portion of the branch, those corresponding to smaller increases in the pressure, are possible end states for weak detonations. By their nature, these waves occur only rarely \cites{FickettDavis,Williams}. Also, in contrast to the case of a strong detonation which, like a classical gas-dynamical shock, is supersonic ahead of the front and subsonic behind, in a weak detonation the gas flow relative to the reaction front is supersonic both ahead of and behind the front\footnote{The CJ detonation is sonic behind the front.}. 
In this paper, we examine the stability of these waves in a simplified combustion model.

The simplified model that we use is an natural extension of Majda's ``qualitative'' model for gas-dynamical combustion \cite{M_SIAMJAM81}. Majda proposed the nonlinear system of partial differential equations,
\begin{subequations}\label{eq:mm0}
\begin{align}
(u+qz)_t&+f(u)_x=B u_{xx}\,, \label{eq:mm0-u}\\
z_t&=-k\varphi(u)z\,;\label{eq:mm0-z}
\end{align}
\end{subequations}
as a model which is, one the one hand, mathematically tractable, and which, on the other hand, is expected to retain important aspects of the strongly coupled interactions between the nonlinear motion of a gas mixture and chemical reactions involving the different species of gas making up the mixture. 
 In equation \eqref{eq:mm0}, the unknown function  $u=u(x,t)$ is real valued and should be thought of as a stand-in for density, velocity, and temperature; the other unknown $z=z(x,t)$ satisfies $0\leq z\leq 1$ and measures the fraction by mass of reactant (fuel) in a simple one-step reaction scheme; the flux $f$ is a nonlinear convex function; $\varphi$ is the ignition function---it turns on the reaction; and $k$, $q$, and $B$ are positive constants measuring reaction rate, heat release, and viscosity, respectively. The main result of Majda's analysis \cite{M_SIAMJAM81} is a proof of the existence of strong and weak detonations, particular kinds of traveling waves, for the system \eqref{eq:mm0}. These waves are combustion waves which connect an unburned state ($z=1$) to a completely burned state ($z=0$); they are analogues of the corresponding waves in classical combustion theory \cite{CourantFriedrichs}. Notably, Majda showed that some of the strong detonations feature a ``spike'' in agreement with the classical theory. Moreover, Majda's proof shows that weak-detonation solutions of \eqref{eq:mm0} exist only for distinguished values of the parameters in agreement with the observation that such waves should be rare. 
Indeed, Majda's construction is explicit, and it shows that the existence of a heteroclinic orbit corresponding to a weak detonation requires the structurally unstable intersection in the plane of the one-dimensional stable manifold at the unburned state with the one-dimensional stable manifold at the burned state. Here, our focus is on the dynamical stability of these waves as solutions of the evolutionary partial differential equation. We note that these waves are \emph{undercompressive}; that is, from the hyperbolic viewpoint, the ``shock'' formed by the end states does not satisfy the Lax shock condition due to a deficit of incoming characteristics. This feature affects the stability analysis. Indeed, in contrast to the case of strong detonations which are of Lax type, we know of no stability results for these waves which are based on energy estimates and/or weighted norms. The outgoing characteristic is an obstacle to these methods. By contrast, our approach, based on the Evans function, applies to such undercompressive waves.

\subsubsection{Model and waves}
To describe our results more precisely, we now introduce the version of the Majda model to which will be the setting for our analysis\footnote{In the introduction of \cite{HLZ_majda}, there is a careful description of a number of the many variations of the Majda model (scalar balance law coupled to reaction equation) that have appeared in the literature since Majda's original paper. We call all of these models, ``Majda models.'' }.
We begin with the Majda model  \cite{LRTZ_JDE07}:
\begin{subequations}\label{eq:mm}
\begin{align}
u_t+f(u)_x& =B u_{xx}+qk\varphi(u)z\,,\label{eq:mm1} \\
z_t & = D z_{xx}-k\varphi(u)z\,.\label{eq:mm2}
\end{align}
\end{subequations}
 Here, the scalar unknown $u$ combines various aspects of density, velocity, and temperature.  The unknown $z\in[0,1]$ is the mass fraction of reactant. The reaction constants, are the heat release $q>0$ and the reaction rate $k>0$.  Here, $q>0$ indicates an exothermic reaction.
The diffusion coefficients $B$ and $D$ are also assumed to be positive constants. We make the standard assumption, following \cite{M_SIAMJAM81}, that $f\in C^2(\RR)$ with
\begin{equation}\label{eq:f-u}
\frac{\dif f}{\dif u}>0\,,\quad \frac{\dif^2 f}{\dif u^2}>0\,.
\end{equation}
We shall use the Burgers flux,
\beq\label{eq:burgers}
f(u)=\frac{u^2}{2}\,,
\eeq
as the nonlinearity in our numerical calculations below, and we shall restrict the state variable $u$ to positive values. Thus, all of our conclusions based on numerical Evans-function computation are restricted to this form of the nonlinearity.  
Finally, we assume that
the ignition function $\varphi$ is given by
\beq\label{eq:experimentalphi}
\varphi(u)=\begin{cases}
0, &\text{if}\; u\leq u_\mathrm{ig}\,, \\
\me^{-E_A/(u-u_\mathrm{ig})}, &\text{if}\; u>u_\mathrm{ig}\,,
\end{cases}
\eeq
where $E_A>0$ is the activation energy and $u_\mathrm{ig}$ is a fixed ignition threshold.
Detonation waves are solutions of the special form
\[
u(x,t)=\bar u(x-st),\quad z(x,t)=\bar z(x-st),\quad s>0,
\]
which satisfy
\[
\lim_{\xi\to+\infty}(\bar u(\xi),\bar z(\xi)) = (u_\sp,z_\sp)=(u_\sp,1)\;\;\text{and}
\lim_{\xi\to-\infty}(\bar u(\xi),\bar z(\xi)) = (u_\sm,z_\sm)=(u_\sm,0)\,.
\]
These waves move from left to right and leave completely burned gas in their wake. As we describe in more detail below, \emph{weak detonations} also satisfy 
\[
f'(u_\spm)<s\,,
\]
and our interest is in the stability, or sensitivity to perturbation, of these waves as solutions of \eqref{eq:mm}.

\subsubsection{Stability}
To describe our approach to determining stability, we denote by $L$ the linear operator obtained by linearizing about the wave in question\footnote{The precise form of $L$ can be seen below in \eqref{eq:majda}.}. Thus, the approximate evolution of a perturbation $v$ is described by a linear equation of the form $(\partial_t-L)v=0$.  The Evans function, denoted by $E$, is an analytic function associated with the operator $L$. Its zeros $\lambda$ with $\re\lambda\geq 0$ correspond to eigenvalues of $L$. As Proposition \ref{prop:lrtz} below shows, the spectral information encoded in the zeros of the $E$ can be used to draw conclusions about the nonlinear stability of the wave in question.
\begin{prop}[Lyng-Raoofi-Texier-Zumbrun \cite{LRTZ_JDE07}]\label{prop:lrtz}For model \eqref{eq:mm} as described above, if the Evans-function condition,
\begin{equation}\label{eq:evans_condition}
\text{$E(\cdot)$ has precisely one zero in $\{\re\lambda\geq 0\}$ (necessarily at $\lambda=0$)\,,}
\tag{$\star$}
\end{equation}
holds, 
 then a weak detonation wave is $\hat L^\infty\to L^p$ nonlinearly phase-asymptotically orbitally stable, for $p>1$. Here,
\begin{equation}\label{eq:l_infty_hat}
\hat L^\infty(\mathbb{R}):=\{f\in\mathscr{S}'(\mathbb{R})\;:\;(1+|\cdot|)^{3/2}f(\cdot)\in L^\infty(\mathbb{R})\}.
\end{equation}
\end{prop}
\br
We recall that if $X$ and $Y$ are Banach spaces, a traveling wave $\bar u$ is $X\to Y$ \emph{nonlinearly orbitally stable} if, given initial data $u_0$ in $X$ such that if $\norm{\bar u-u_0}{X}$ is sufficiently small, there is a phase shift $\delta=\delta(t)$ such that $\norm{u(\cdot,t)-\bar u(\cdot-\delta(t),t)}{Y}\to 0$ as $t\to \infty$. If also $\delta(t)$ converges to a limiting value $\delta(+\infty)$, the wave is \emph{nonlinearly phase-asymptotically orbitally stable}.
\er
The proof of Proposition \ref{prop:lrtz} is based on the pointwise Green-function techniques developed by Zumbrun and collaborators; see, e.g., \cite{ZH_IUMJ98}. Briefly,
if one is able to obtain sufficient estimates on the Green function $G(x,t;y)$ solving  $(\partial_t-L)G=\delta_y$, it is possible to close a iterative argument to establish a result like Proposition \ref{prop:lrtz}. The main work of \cite{LRTZ_JDE07} is devoted to establishing such bounds under the assumption that condition \eqref{eq:evans_condition} holds. Thus, our primary purpose here is to locate the unstable  zeros (if any) of the Evans function. In this paper we restrict our attention to the case of weak detonations; a parallel Evans-based stability analysis for strong detonations has been done \cite{HLZ_majda}. Because, in all but the most trivial cases, the Evans function is typically too complex to be computed analytically, our approach is based on the combination of an energy estimate to eliminate the possibility of large unstable zeros and the numerical approximations of the Evans function to deal with the remaining, bounded region of the unstable complex plane. As we describe below, a particular challenge associated with performing Evans-function computations for weak-detonation waves is finding the distinguished parameter values for which these waves exist; see \S\ref{ssec:approx} below for further discussion and more details about this issue.

\subsection{Related work: stability, weighted norms \& energy estimates}
There are a number of stability results for strong-detonation wave solutions of various incarnations of  the Majda model that have been obtained directly by combinations of energy estimates, spectral analysis, and weighted norms. A detailed overview of these results can be found in the recent work of Humpherys et al.\ \cite{HLZ_majda}. 
Notably, however,  outside of the Evans-function framework, namely \cites{LZ_PD04,LRTZ_JDE07}, we know of \emph{no} stability results for weak-detonation solutions of the Majda model.  

There are results for weak-detonation solutions of the closely related Rosales-Majda model \cite{RM_SIAMJAM83}:
\begin{subequations}\label{eq:mr}
\begin{align}
u_t&+\left(\frac{u^2}{2}-Qz\right)_x=B u_{xx}\,,\\
z_x&= K\varphi(u)z\,.
\end{align}
\end{subequations}
This model was extracted from the physical equations in the Mach ``$1+\epsilon$'' asymptotic regime by Rosales and Majda. That is, this simplified model describes detonation waves which propagate with a speed close to the sound speed, and the model is expected to capture some of the crucial interactions between the nonlinear gas-dynamical motion of the gas mixture and the chemical reactions. 
Liu \& Yu~\cite{LY_CMP99} and Szepessy~\cite{S_CMP99} have both treated the stability of weak-detonation solutions of \eqref{eq:mr}.

\subsection{Outline}
In \S\ref{sec:prelim} we review the existence problem for strong and weak detonations. Because weak detonations are a structurally unstable phenomenon, we also describe our numerical procedure for approximating these waves.   In \S\ref{sec:spectral}, we set up the spectral stability problem, and we describe the construction of the Evans function, and our algorithm for approximating the Evans function and locating its zeros. We also establish, by means of an energy estimate, an upper bound on the moduli of possible unstable eigenvalues. This limits possible unstable zeros of the Evans function  to a bounded region of the complex plane. The final sections, \S\ref{sec:experiments} and \S\ref{sec:discussion}, contain descriptions, results, and interpretations of our numerical experiments.

\section{Preliminaries}\label{sec:prelim}

\subsection{The profile existence problem}\label{ssec:connections}
\subsubsection{Basic analysis}
As noted above, we seek solutions of \eqref{eq:mm} of the form
\beq\label{eq:twa}
u(x,t)=\bar u(x-st),\quad z(x,t)=\bar z(x-st),\quad s>0,
\eeq
of \eqref{eq:mm} which satisfy
\[
\lim_{\xi\to+\infty}(\bar u(\xi),\bar z(\xi)) = (u_\sp,z_\sp)=(u_\sp,1)\;\;\text{and}
\lim_{\xi\to-\infty}(\bar u(\xi),\bar z(\xi)) = (u_\sm,z_\sm)=(u_\sm,0)\,.
\]
These waves move from left to right and leave completely burned gas in their wake.
Thus, after dropping the bars, we see that the ansatz \eqref{eq:twa} leads from \eqref{eq:mm} to the system of ordinary differential equations,
\begin{subequations}
\begin{align}
-su'+f(u)'& =B u''+qk\varphi(u)z\,,\label{eq:twmm1} \\
-sz'& = D z''-k\varphi(u)z\,.\label{eq:twmm2}
\end{align}
\end{subequations}
 where $'$ denotes differentiation with respect to the variable $\xi:=x-st$. After a simple algebraic rearrangement, we can integrate \eqref{eq:twmm1}, and we obtain, finally, the first-order system
\begin{subequations}\label{eq:tw}
\begin{align}
u'&=B^{-1}\big(f(u)-f(u_\sm)-s(u-u_\sm)-q(sz+D y)\big)\,, \label{eq:tw1}\\
z'&=y\,,\label{eq:tw2}\\
y'&=D^{-1}\big(-sy+k\varphi(u)z\big)\,. \label{eq:tw3}
\end{align}
\end{subequations}
In \eqref{eq:tw}, we have written $y:=z'$ to express the system in first order. We sometimes write this system compactly as $U'=F(U)$ with $U=(u,z,y)^\mathrm{t}$, and we write $A(U)=\dif F(U)$.
We require  that $u_\spm$ satisfy
\begin{equation}
u_\mathrm{ig}<u_\sm
\quad\text{and}\quad
u_\sp<u_\mathrm{ig}\,,
\label{eq:u_plus}
\eeq
so that
\beq
\varphi(u_\sm)>0\,,\quad \varphi(u_\sp)=0\,,\quad \varphi'(u_\sp)=0\,.
\label{eq:phi_endstates}
\end{equation} \noindent
Equation~\eqref{eq:u_plus} is needed so that the unburned state $U_\sp=(u_\sp,1,0)$ is an equilibrium for the traveling-wave equation $U'=F(U)$. Indeed, to guarantee that both $U_\spm$ are equilibria, we evidently require the Rankine-Hugoniot condition
\begin{equation}
f(u_\sp)-f(u_\sm)=sq+s(u_\sp-u_\sm)\,,
\label{eq:rh}
\tag{RH}
\end{equation}
together with the requirements that $y_\spm=0$ and
\(
k\varphi(u_\spm)z_\spm=0
\).
We shall also make use of the convenient shorthand $a_\spm:=f'(u_\spm)$.
If $u_\sp<u_\sm$, the combustion wave is a \emph{detonation}., and detonations are classified as of \emph{strong}, \emph{weak}, or \emph{Chapman-Jouguet} type according to the relationship between $a_\spm$ and the wave speed $s$; see Table \ref{tab:class}.
\begin{table}[ht]
\begin{center}
\begin{tabular}{|l|l|} \hline
Strong  & $a_\sm>s>a_\sp$ \\
Weak  &  $s>a_\sm,a_\sp$ \\
Chapman--Jouguet  &  $a_\sm=s>a_\sp$ \\
\hline
\end{tabular}
\end{center}
\caption{Classification of detonation waves.}
\label{tab:class}
\end{table}
 In this paper, we focus on weak detonations.
\begin{figure}[t] 
   \centering
   \includegraphics[width=9cm]{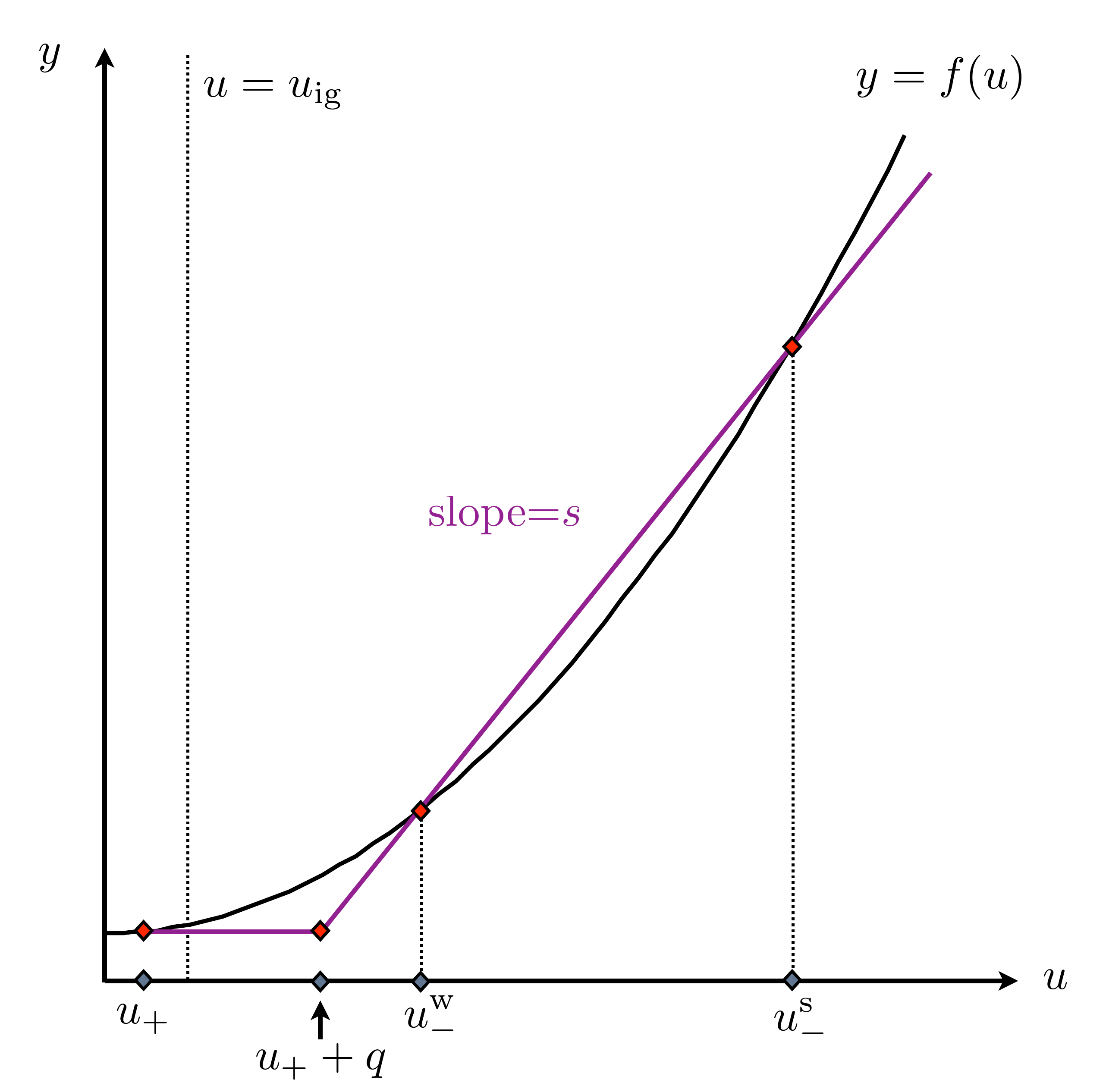}    \caption{The CJ diagram.}
   \label{fig:cj_diagram1}
\end{figure}

The first step in constructing detonation waves is to identify all the possible equilibria of \eqref{eq:tw}. This amounts to solving \eqref{eq:rh}. The structure of solutions is well known.
\begin{prop}[\cites{M_SIAMJAM81,LRTZ_JDE07}]\label{prop:RHexist}

Fix $u_\sp$. Then, there are 0, 1, or 2 solutions of \eqref{eq:rh} depending on the wave speed $s$. In particular,  there is a speed $s^\mathrm{cj}$ depending on $u_\sp$ such that the following holds.
\begin{enumerate}
\item[(a)] For $s< s^\mathrm{cj}$, there exist no solutions $u_\sm>u_\sp$.
\item[(b)] For $s=s^\mathrm{cj}$, there exists one solution $u_\sm^\mathrm{cj}$ (Chapman--Jouguet detonation).
\item[(c)] For $s>s^\mathrm{cj}$, there
exist two states $u^\mathrm{s}_\sm>u_\sm^\mathrm{w}>u_\sp$ for which \eqref{eq:rh}
(but not necessarily \eqref{eq:phi_endstates})
is satisfied (weak and strong detonation).
\end{enumerate}
See \textsc{Figure}~\ref{fig:cj_diagram1}.
\end{prop}

The case of our principal interest is that the wave is a weak detonation. That is,
\beq\label{eq:weakdet}
s>a_\sp,a_\sm\,.
\eeq
We assume, then, for the remainder of the paper that \eqref{eq:weakdet} holds.
Linearizing \eqref{eq:tw} around the state $U^\mathrm{w}_\sm:=(u_\sm^\mathrm{w},z_\sm,y_\sm)=(u_\sm^\mathrm{w},0,0)$, we find the system of ordinary differential equations $U'=A(U_\sm^\mathrm{w})U$ with 
\begin{equation}
A(U_\sm^\mathrm{w})=\begin{bmatrix}
B^{-1}(a_\sm-s) & B^{-1}(-sq) & qB^{-1}D \\
0 & 0 & 1 \\
0 & kD^{-1}\varphi(u_\sm^\mathrm{w}) & -sD^{-1}
\end{bmatrix}
\,.
\label{eq:minus_lin_twode}
\end{equation}
The coefficient matrix is upper block-triangular, and it is immediate that it has one positive eigenvalue and two negative eigenvalues. Thus, there is a one-dimensional unstable manifold at $U_\sm^\mathrm{w}$. Similarly, we compute directly that
\begin{equation}
A(U_\sp)=
\begin{bmatrix}
B^{-1}(a_\sp-s) & B^{-1}(-sq) & qB^{-1}D \\
0 & 0 & 1 \\
0 & 0 & -sD^{-1}
\end{bmatrix}\,.
\label{eq:plus_lin_twode}
\end{equation}
Again using the structure of the coefficient matrix, we see immediately that there are two negative eigenvalues and one zero eigenvalue. It is straightforward to see in this case that the center manifold is a line of equilibria, so no orbit may approach the rest point $U_\sp$ along the center manifold. This follows from the nature of the ignition function $\varphi$. Since no trajectory can approach the unburned state along the center manifold, a connection corresponding to a weak detonation corresponds to the intersection of the one-dimensional unstable manifold exiting the burned end state with the two-dimensional stable manifold entering the unburned state in the phase space $\mathbb{R}^3$.
\br[Strong Detonations]
Repeating the above calculation in the case that $a_\sm>s>a_\sp$, we see immediately that a strong-detonation connection, by contrast, corresponds to the structurally stable intersection of a pair of two-dimensional manifolds in $\mathbb{R}^3$. See \cite{HLZ_majda} for the examination of the Evans condition \eqref{eq:evans_condition}  in the case of strong detonations.
\er
The next lemma is immediate by the bounds of the stable (unstable) manifold theorem.
\begin{lemma}\label{lem:expdecay}
Traveling-wave profiles $(\bar u, \bar z)$ corresponding to weak or strong
detonations satisfy, for some $C>0$, $\theta>0$,
\beq
\label{eq:expdecay}
\left|(\dif/\dif \xi)^k \Big((\bar u, \bar z)(\xi)- (u,z)_\spm\Big)\right|\le Ce^{-\theta |\xi|},
\qquad
\xi\gtrless 0, \quad 0\le k\le 3\,.
\eeq
\end{lemma}

\subsubsection{End states and parametrization}
Suppose $\big( \bar u(\xi),\bar z(\xi)\big)$ is a traveling-wave profile of \eqref{eq:mm1}--\eqref{eq:mm2} satisfying \eqref{eq:weakdet}.
Evidently, $(\bar u,\bar z)$ is a steady solution of
\begin{subequations}\label{eq:mmm}
\begin{align}
u_t-su_x+(u^2/2)_x& =B u_{xx}+qk\varphi(u)z\,,\label{eq:mmm1} \\
z_t -sz_x& = D z_{xx}-k\varphi(u)z\,.\label{eq:mmm2}
\end{align}
\end{subequations}
As a preliminary step, we rescale space and time via
\beq\label{eq:scale}
\tilde{x} = \dfrac{s}{B} x\,,\quad\tilde{t} = \dfrac{s^2}{B}t\,;
\eeq
 we also rescale $u$ so that
\beq\label{eq:scale2}
s\tilde u(\tilde x,\tilde t)=u(x,t)\quad\text{and}\quad\tilde z(\tilde x,\tilde t) = z(x,t)\,.
\eeq
In the new scaling equation \eqref{eq:mmm} takes the form
\begin{equation*}
\tilde{u}_{\tilde{t}} - \tilde{u}_{\tilde{x}} + \left(\tilde{u}^2/2\right)_{\tilde{x}} = \tilde{u}_{\tilde{x} \tilde{x}} + \tilde{q} \tilde{k} \tilde{\varphi}(\tilde{u}) \tilde{z}\,,
\quad
\tilde{z}_{\tilde{t}} - \tilde{z}_{\tilde{x}} = \tilde{D} \tilde{z}_{\tilde{x} \tilde{x}} - \tilde{k} \tilde{\varphi}(\tilde{u}) \tilde{z}\,,
\end{equation*}
where $\tilde{k} = kB/s^2$, $\tilde{\varphi}(\tilde{u}) = \varphi\left(\tilde{u}/s\right)$, $\tilde{q} = q/s$, and $\tilde{D} = D/B$.
We omit the tildes from this point forward, and focus on the system
\begin{subequations}\label{eq:majda_simp}
\begin{align}
u_t - u_x + \left(u^2/2\right)_x &= u_{x x} + q k \varphi(u) z\,,\\
z_t - z_x &= D z_{x x} - k \varphi(u) z\,.
\end{align}
\end{subequations}
The scaling analysis shows that we can take $s=1$ and the viscosity coefficient $B=1$.  In this case, \eqref{eq:rh} reduces to
\begin{equation}
\frac{1}{2}(u_\sp^2-u_\sm^2) = u_\sp - u_\sm + q = 0\,. \label{eq:RHsimp}
\end{equation}
Consequently, we can solve for the burned state $u_\sm$ in terms of $q$ and $u_\sp$:
\begin{equation}
u_\sm = 1 - \sqrt{1-2(q+u_\sp(1-u_\sp/2))}\,. \label{eq:uminus}
\end{equation}
Therefore, the physical range for $q,u_\sp$ is
\begin{equation}
\mathcal{U} := \left\{(u_\sp,q) \in \mathbb{R}^2 \,|\, 0\leq u_\sp\leq u_\sm, 0\leq q\leq \frac{1}{2}(u_\sp-1)^2\right\}. \label{eq:uplusq}
\end{equation}

\subsubsection{Profile properties}\label{ssec:properties}
It is worth noting that for $u < \ui$, we can solve the system \eqref{eq:tw} explicitly by a simple integration since $\varphi(u) = 0$ for $u < \ui$.  In this case we find
\begin{subequations}\label{eq:end_states}
\begin{align}
&u(\xi) = 1+ \beta \tanh(-\beta \xi + C)\,, \label{eq:uprofile} \\
&z(\xi) = 1 - C D \me^{-\xi/D}\,, \\
&y(\xi) = C \me^{-\xi/D}\,,
\end{align}
\end{subequations}
where $\beta := \sqrt{u_\sm^2 - 2u_\sm + 2q + 1}$ and $C$ is a constant of integration.  We see, by inspection of the formula \eqref{eq:uprofile} that $u$ is monotone for $u<u_{\mathrm{ig}}$.  We shall now show that this monotonicity persists for $u>\ui$, and we shall use this monotonicity property in  Proposition \ref{prop:hfb} below. 
\br
Majda's construction, by phase plane analysis \cite{M_SIAMJAM81}, of weak detonation profiles for the model \eqref{eq:mm0} ($D=0$) shows that $\bar u$ is monotone in $\xi$. We note that, as discussed in \cite{HLZ_majda}, the traveling-wave equation for \eqref{eq:mm} ($D\neq0$) no longer has a planar phase space, and the resulting dynamics are substantially more complicated. Nonetheless, we are able to give a new argument establishing monotonicity for weak-detonation profiles $\bar u$ in the case that $D\neq 0$.
\er
\begin{prop}[Monotonicity of $\bar u$]\label{prop:mono}
For any weak detonation profile, $\bar u$ is decreasing in the wave variable $\xi$.
\end{prop}
\begin{proof}
For notational convenience, we omit the bars on solutions of the traveling-wave equation. We define
\begin{equation}\label{eq:u_comp}
\Phi(u) :=\frac{1}{2}(u^2-u_\sm)-(u-u_\sm) = (u-u_\sm)\left(\frac{1}{2}(u+u_\sm)-1\right)\,,
\end{equation}
so that equation \eqref{eq:tw1} can be written
\begin{equation*}
u' = \Phi(u) - q(z+Dy)\,.
\end{equation*}
First we claim that for weak detonations $u(\xi) \leq u_\sm$ for all $\xi$.  Suppose not.  Then since $u_\sm < 1$ for weak detonations, there exists a $\xi$ such that $u_\sm < u(\xi) < 1$.  In this case we see that $\Phi(u(\xi)) < 0$.
Notice also that
\begin{equation*}
(z+Dy)' = k\varphi(u)z \geq 0\,.
\end{equation*}
Thus, $q(z(\xi)+Dy(\xi)) \geq 0$.  Combining these facts we see that
\begin{equation*}
u'(\xi) = \Phi(u(\xi)) - q(z(\xi)+Dy(\xi))<0\,,
\end{equation*}
 a contradiction.

Consequently, there must exist some $L$ such that $u'(\xi)<0$ for all $\xi \leq L$.  If $u$ is increasing for some $\xi > L$, then there exists $\xi_1<\xi_2$ such that such that $u(\xi_1) = u(\xi_2)=u_*$.  However, $u'(\xi_1)<0$.  So
\begin{equation*}
\Phi(u_*) < q(z+Dy)(\xi_1)\,.
\end{equation*}
Since $z+Dy$ is nondecreasing in $\xi$ this implies $\Phi(u_*)<q(z+Dy)(\xi_2)$.  Therefore $u'(\xi_2) <0$, a contradiction.  The desired result follows.
\end{proof}
\begin{remark}
It is worth noting the implications of the proof for strong detonations.  Strong detonation profiles are not all monotone, the difference being that for strong detonations $u_\sm > 1$.  Consequently the first claim in the proposition does not hold.  The proof does show that for strong detonations a profile cannot move from decreasing to increasing moving left to right.  Thus all non-monotonic strong detonation profiles are of the form $u'(\xi) \geq 0$ for $\xi \leq L$ for some $L$, and $u'(\xi) \leq 0$ for all $\xi \geq L$.
\end{remark}

\subsection{Numerical approximation of profiles}\label{ssec:approx}
\subsubsection{Existence}
In the case that \eqref{eq:weakdet} holds, we have seen that the existence of a weak detonation requires that the intersection of the two-dimensional stable manifold
\(
W^\mathrm{s}(U_\sp)
\)
and the one-dimensional unstable manifold
\(
W^\mathrm{u}(U^{\mathrm{w}}_\sm)
\)
form a heteroclinic orbit in $\RR^3$.  We expect that this will only occur for distinguished values of the parameters. Thus, our numerical method for approximating the profile (a necessary step towards the  computation of the Evans function) must incorporate a method for determining these parameters.

Beyn describes a method for dealing with this issue of structural instability in \cite{B_IMAJNA90}.  Essentially, we stabilize the system by introducing the parameter $k$ into it as an unknown satisfying the equation $k' = 0$; we call this the inflation parameter.  This will increase the dimension of both the stable and unstable manifolds.  Consequently after inflating the state space with the parameter $k$, we now seek solutions that are the structurally stable intersection of the three dimensional stable manifold at positive infinity and the two-dimensional unstable manifold and negative infinity in $\RR^4$.
In doing so we lose control over the value of $k$ when finding solutions.  For given values of the other parameters, the solver will return a value of $k$, if a solution can be found.

\subsubsection{Numerical Implementation}
In order to obtain numerical solutions, we use projective boundary conditions at $\pm M$ which are given by $\Pi_\pm(U(\pm M) - U_\spm) = 0$, where $\Pi_\pm$ is the matrix whose columns are orthonormal vectors spanning $E^\mathrm{s}(U_\sm^\mathrm{w})^\perp$ and $E^\mathrm{u}(U_\sp)^\perp$ respectively.  Because the translation of any solution gives another solution, we also employ a phase boundary condition $u(0) = \frac{1}{2}(u_\sp - u_\sm)$.

The result is a three point boundary value problem.  Because most numerical packages are not built to solve a three-point problem, we transform the problem to a two point boundary value problem by doubling the dimension of the problem and halving the domain.  Thus we move from the system
\begin{equation}
\begin{bmatrix}
u'\\
z'\\
y'\\
k'\\
\end{bmatrix}
= \check U' = \check F(U) , \quad x\in [-M,M]\,,
\end{equation}
to the system
\begin{equation}\label{eq:double}
\begin{bmatrix}
\check U'\\
\check V'\\
\end{bmatrix}
=
\begin{bmatrix}
\check F(U)\\
-\check F(V)\\
\end{bmatrix}
, \quad x\in[0,M]\,.
\end{equation}
In \eqref{eq:double} we include three matching boundary conditions of the form $\check U(0) = \check V(0)$.  We now have a two point boundary value problem which we solve using the \textsc{MatLab} package \texttt{bvp6c}, a sixth order collocation method utilized by the Evans function package \textsc{StabLab}. (See \cite{STABLAB}).

Summarizing, we have 1 phase condition, 3 projective conditions, and 4 matching conditions which matches the 8 variables in the transformed system \eqref{eq:double}.  In order to obtain good approximations to weak-detonation profiles, we must supply the solver with a suitable initial guess. This is necessary to start the Newton iteration.
We find that a rudimentary guess that satisfies the appropriate boundary conditions will suffice for ``tame'' values of the parameters.  However, in many regions of parameter space, this simple approach is insufficient.  We rely on continuation in such regions.  That is, we begin in a conservative parameter region and use the solution for a set of parameters as the initial guess for new parameter values.  In this manner we use successive solutions to move to more extreme parameter values.

\subsubsection{Numerical Profile Results}
Even with the use of continuation there are 
parameter regions for which we are unable to obtain solutions or for which solutions do not exist.  In particular we computed solution profiles for values
\begin{equation*}
(D,E_A,q) \in [10^{-3},15] \times [10^{-3},6] \times [.25,.499].
\end{equation*}
We also tested values varying values of $u_\sp$ but found found no qualitative difference and consequently fixed $u_\sp = 0$.  The solutions we were unable to compute in the corners of this parameter space often correspond to extreme, large values of $k$ determined by the solver.

In \textsc{Figure}~\ref{fig:profiles} we display some examples of the numerically computed solutions of \ref{eq:tw}.  We note in particular that small values of $E_A$ result in a large left tail as seen in \textsc{Figure}~\ref{fig:profiles}\subref{fig:Small_E_prof}, while large values of $D$ result in a long right tail as seen in \textsc{Figure}~\ref{fig:profiles}\subref{fig:Large_D_prof}.  In these cases we used continuation and expanded the computational domain in order for the solutions to be within the prescribed tolerance ($10^{-3}$) of the correct limiting values.

\begin{figure}
\centering
\subfloat[]{
\includegraphics[width=0.4\textwidth]{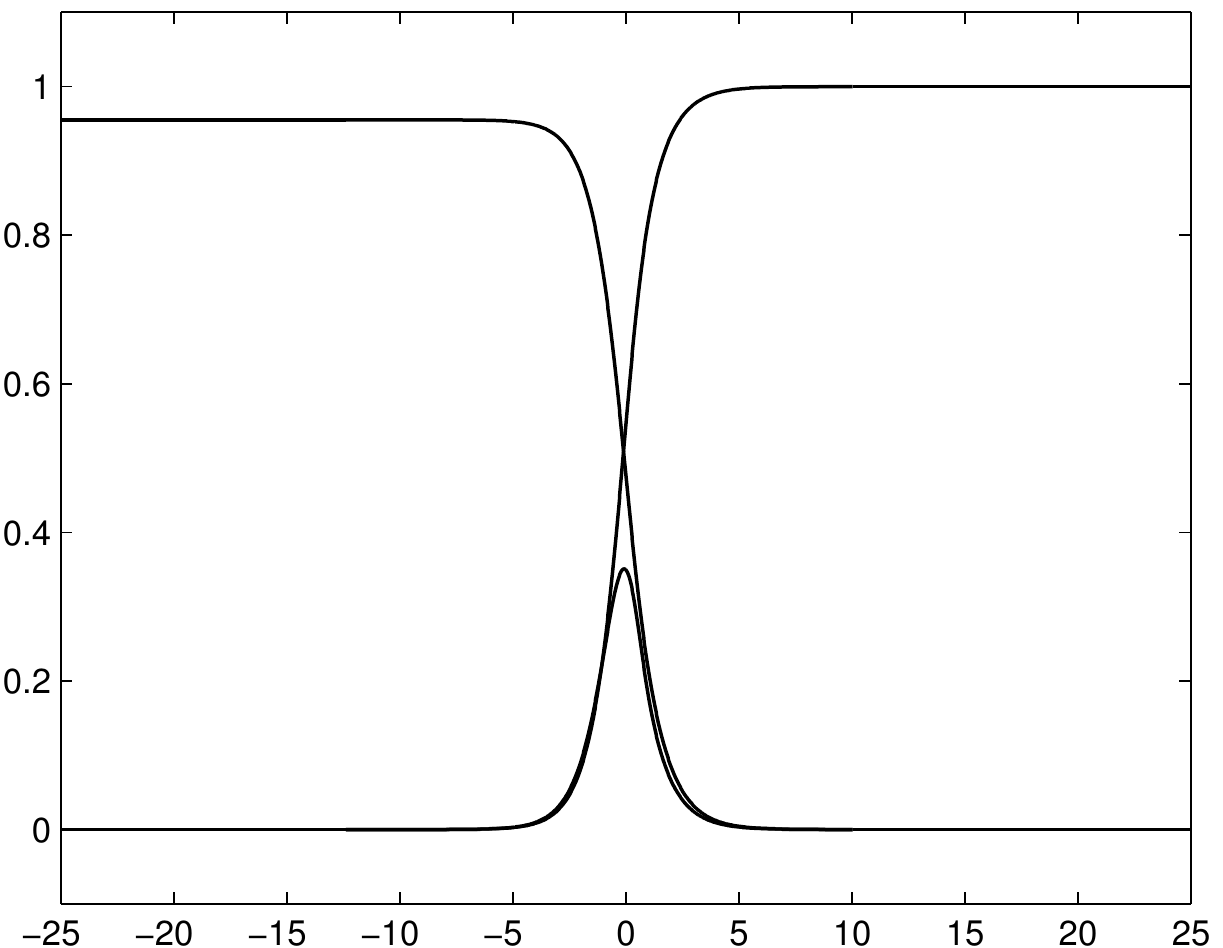}
\label{fig:Large_q_prof}}
\subfloat[]{
\includegraphics[width=0.4\textwidth]{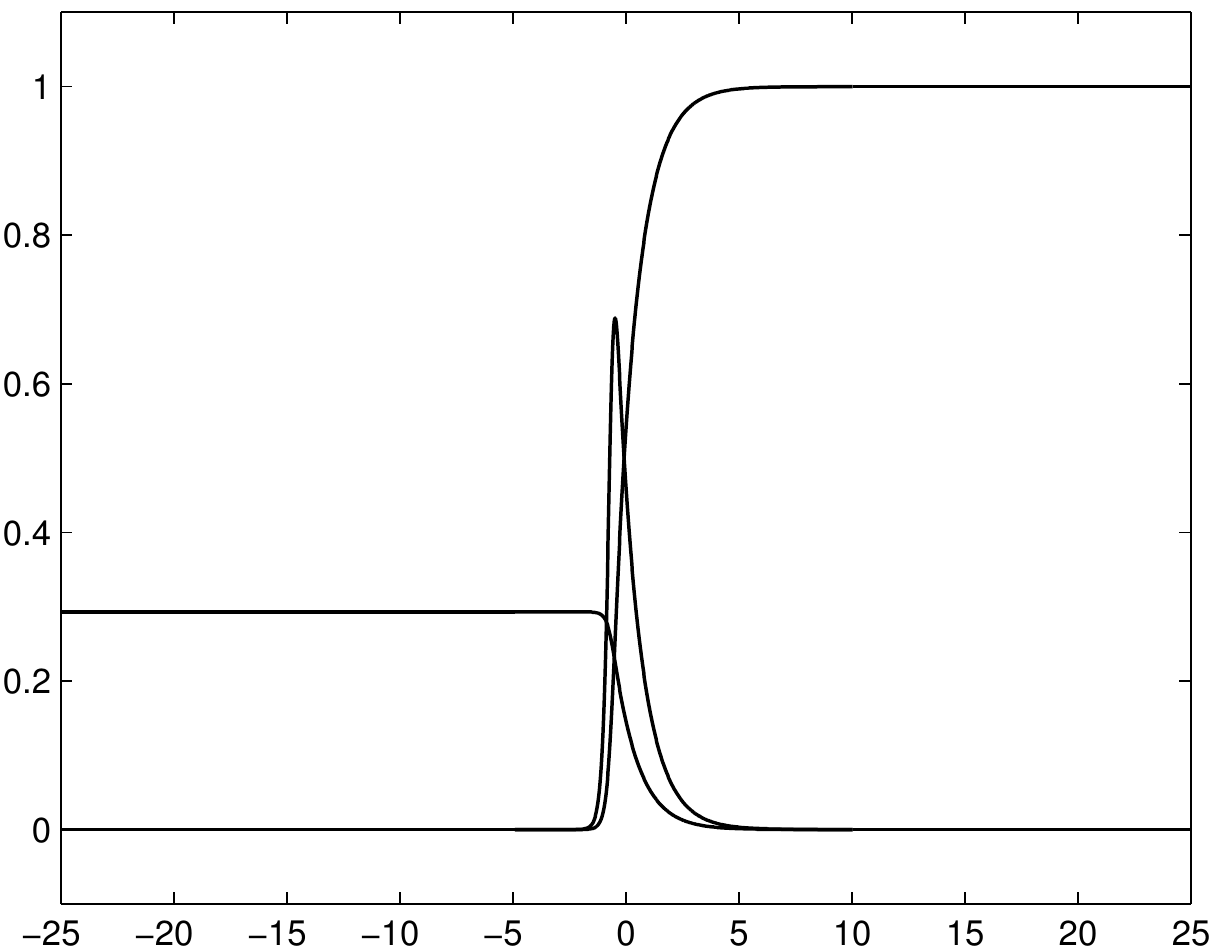}
\label{fig:Small_q_prof}}\\
\subfloat[]{
\includegraphics[width=0.4\textwidth]{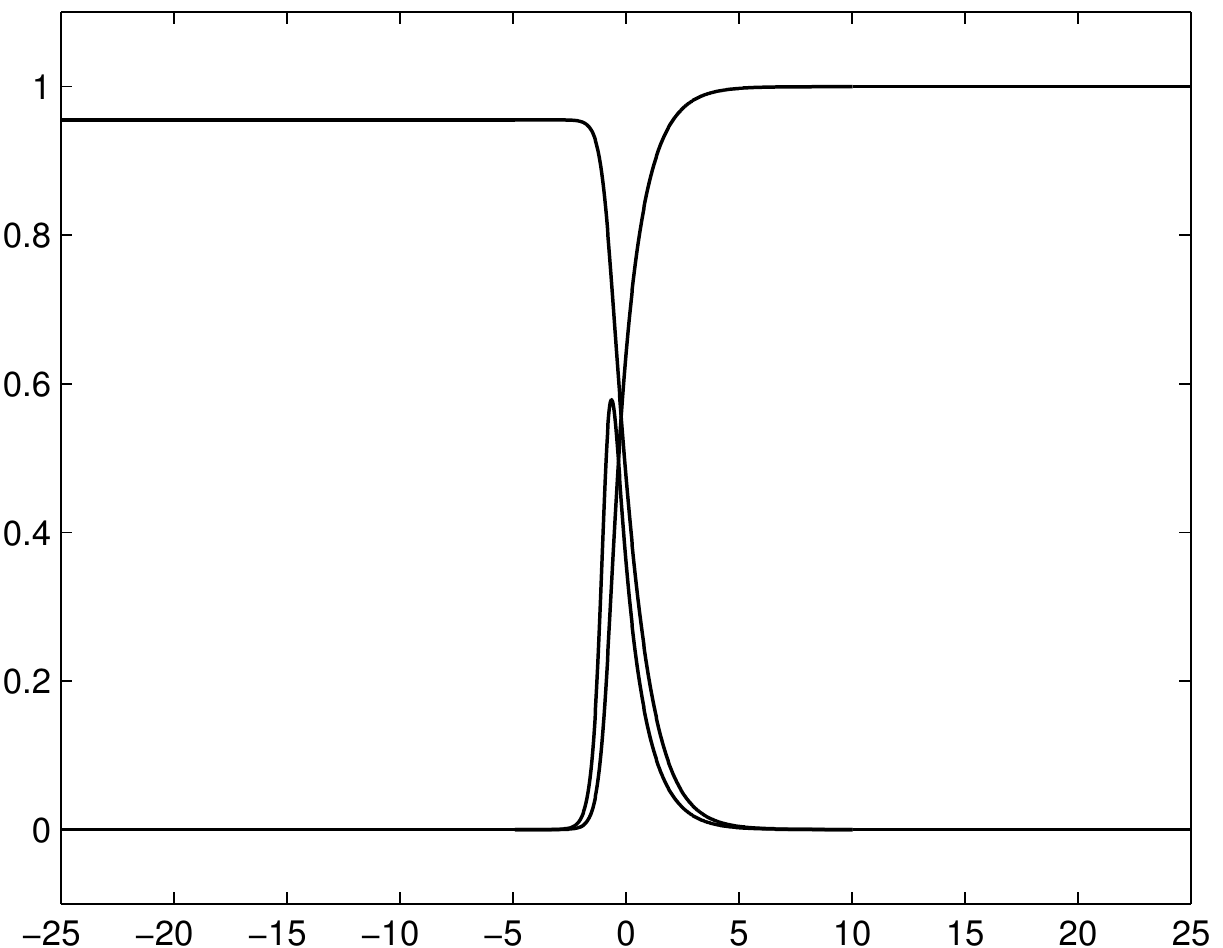}
\label{fig:Large_E_prof}}
\subfloat[]{
\includegraphics[width=0.4\textwidth]{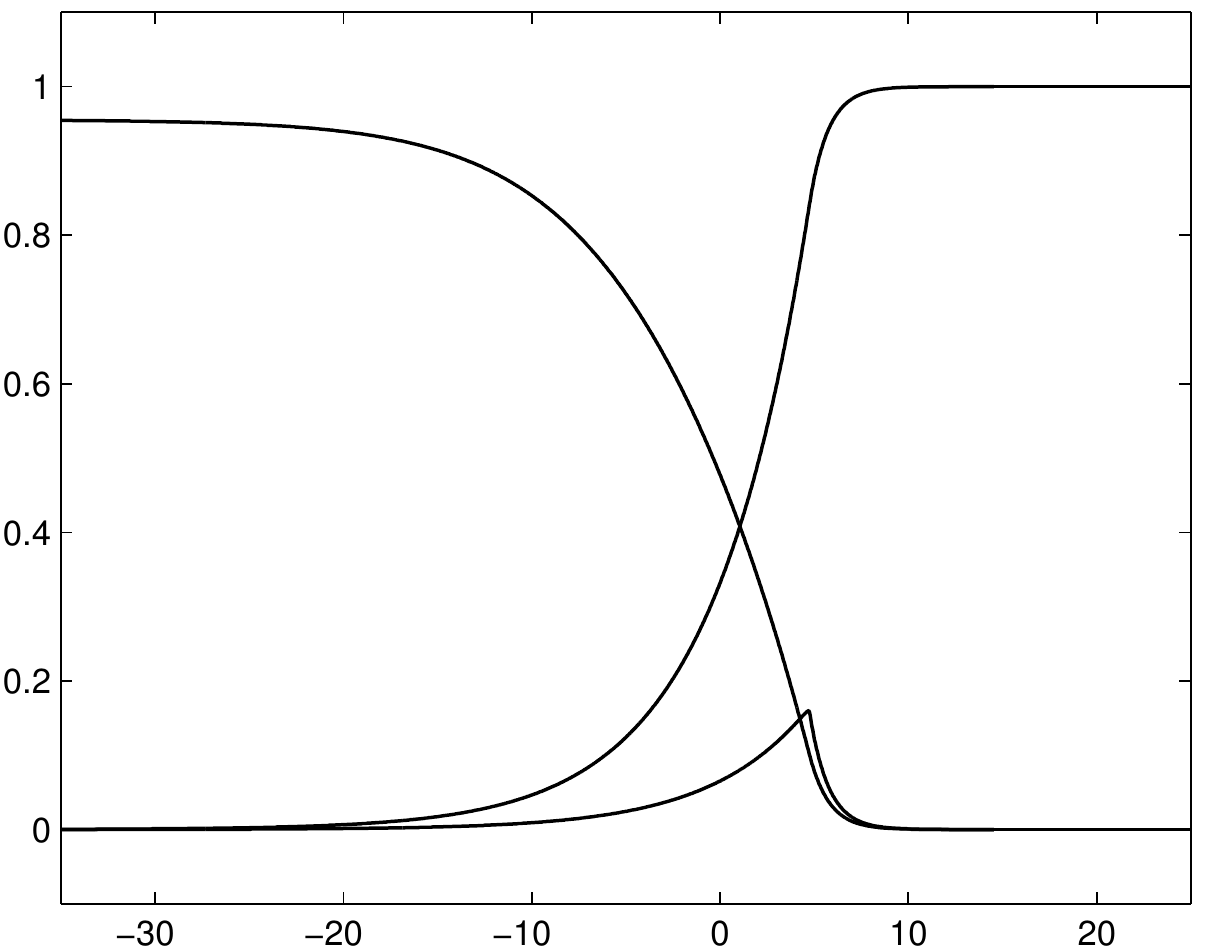}
\label{fig:Small_E_prof}}\\
\subfloat[]{
\includegraphics[width=0.4\textwidth]{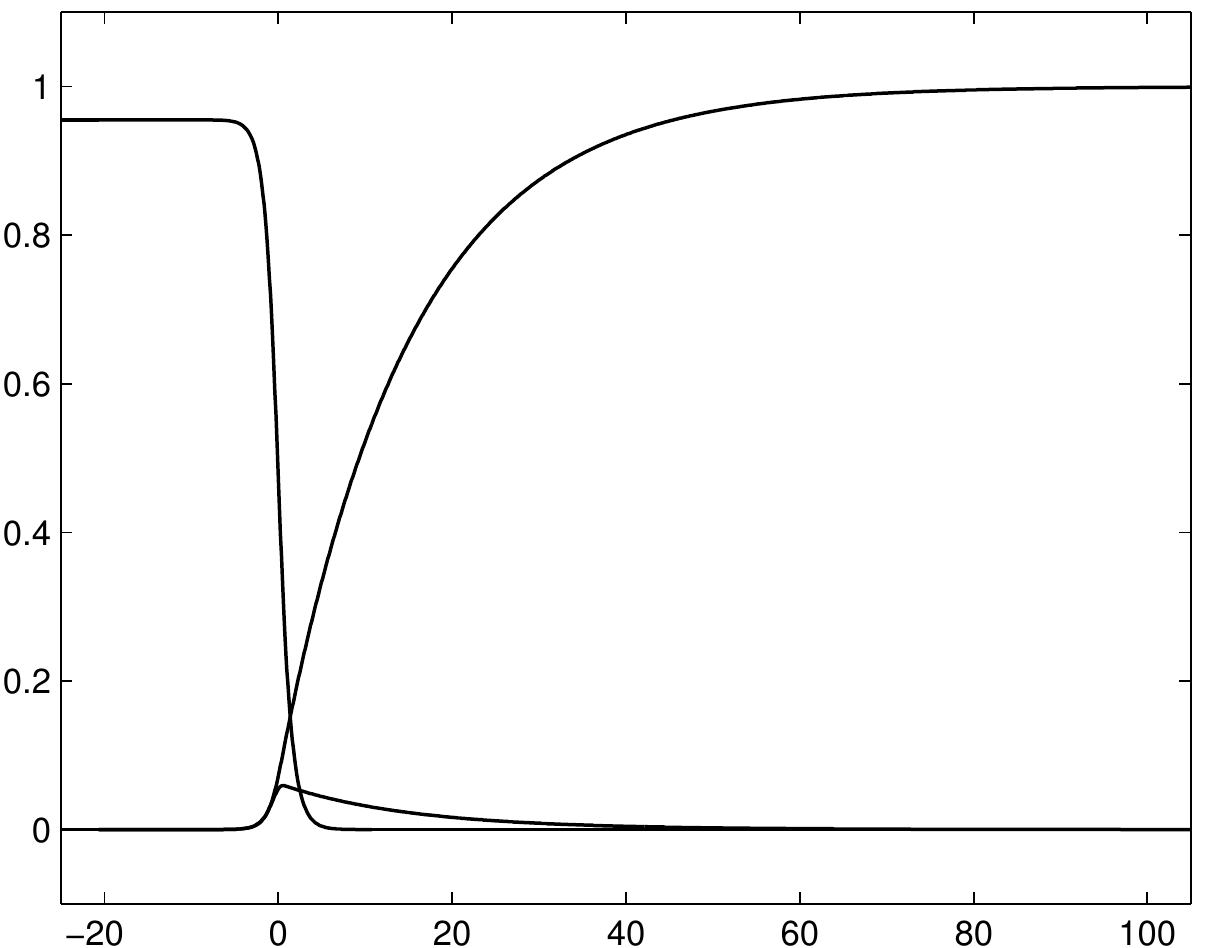}
\label{fig:Large_D_prof}}
\subfloat[]{
\includegraphics[width=0.4\textwidth]{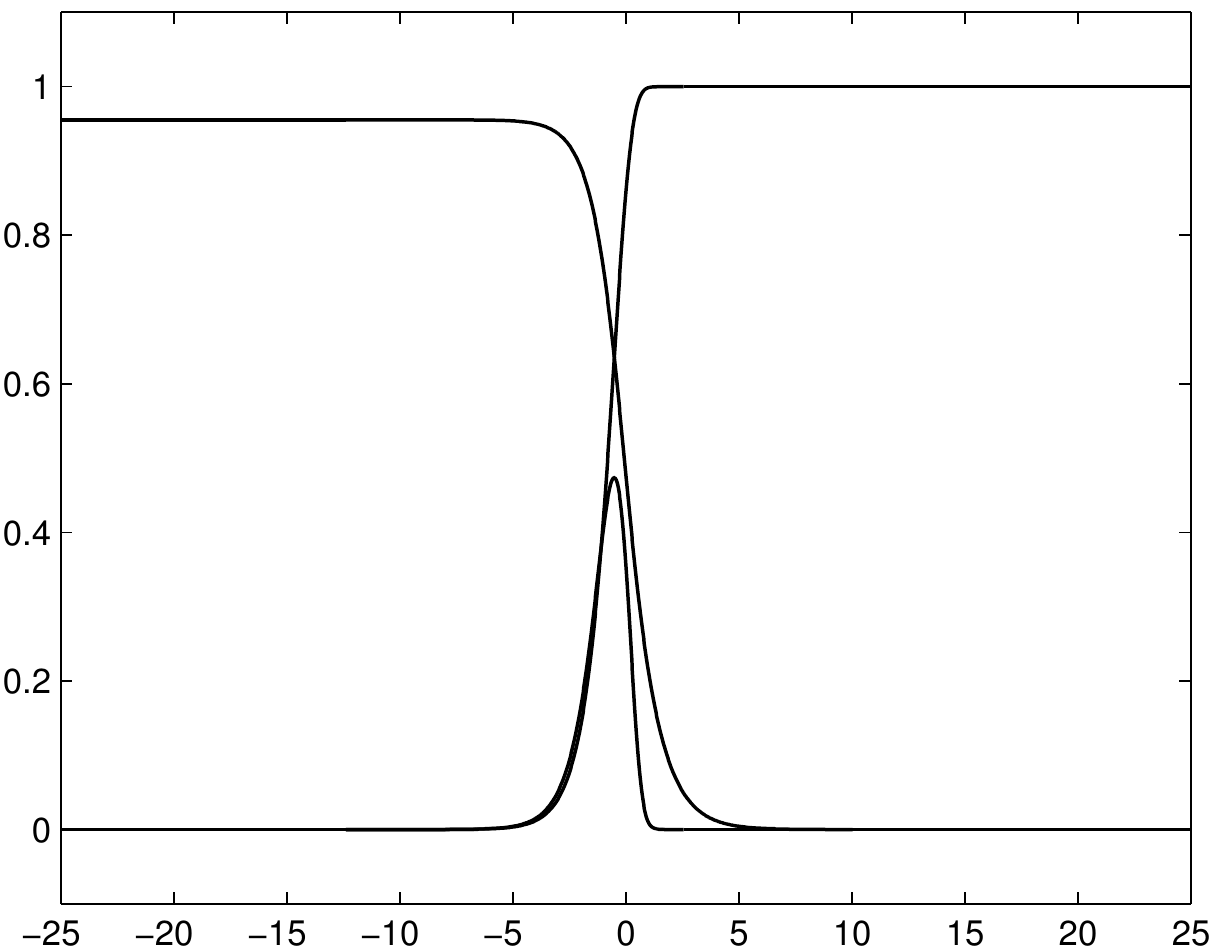}
\label{fig:Small_D_prof}}
\caption{Weak detonation profiles for different parameter values. First we consider the intermediate parameter regime ($D=1$, $\Ea = 1$) for \protect\subref{fig:Large_q_prof} large $q = 0.499$ and \protect\subref{fig:Small_q_prof} small $q = 0.250$.  For large $q$, we also consider \protect\subref{fig:Large_E_prof} large $\Ea$ ($D=1$, $\Ea = 6$, and $q = 0.499$), \protect\subref{fig:Small_E_prof} small $\Ea$ ($D=1$, $\Ea = 10^{-3}$, and $q = 0.499$), \protect\subref{fig:Large_D_prof} large $D$ ($D=15$, $\Ea = 1$, and $q = 0.499$), and \protect\subref{fig:Small_D_prof} small $D$ ($D = 10^{-3}$, $\Ea = 1$, and $q = 0.499)$.
Profiles for $u,y$, and $z$ can be distinguished by noting that $u_\sm > 0$, $y_\sm = y_\sp = 0$, and $z_\sp > 0$.   }
\label{fig:profiles}
\end{figure}

\section{Spectral stability}\label{sec:spectral}
\subsection{Linearized equations \& eigenvalue problem}\label{ssec:eval}
To construct the Evans function, we begin by examining the linearization of \eqref{eq:majda_simp} about the steady solution $(\bar u, \bar z)$.
The linearization is
\begin{subequations}\label{eq:majda}
\begin{align}
&u_t-q(k\varphi'(\bar u)u\bar z+k\varphi(\bar u)z)+((\bar u -1) u)_x=u_{xx},\label{eq:mmlin1}\\
&z_t-z_x=-k\varphi'(\bar u)u\bar z-k\varphi(\bar u)z+D z_{xx},\label{eq:mmlin2}
\end{align}
\end{subequations}
where $u$, $z$ now denote perturbations. Evidently, the corresponding eigenvalue equations are
\begin{subequations}\label{eq:eval}
\begin{align}
&u''=\lambda u-q(k\varphi'(\bar u)u\bar z+k\varphi(\bar u)z)+((\bar u-1) u)'\,,\label{eq:eval1} \\
&z''=D^{-1}\big(\lambda z-z'+k\varphi'(\bar u)u\bar z+k\varphi(\bar u)z\big)\,.\label{eq:eval2}
\end{align}
\end{subequations}
In \eqref{eq:eval} and hereafter $'=\dif/\dif x$.
Alternatively, upon substituting $D z''-\lambda z+z'=k\varphi'(\bar u)u\bar z + k\varphi(\bar u)z$ from \eqref{eq:eval2} into \eqref{eq:eval1}, we can rewrite \eqref{eq:eval1} as
\begin{equation}
u''=\lambda(u+qz)-qz'-qD z''+((\bar u-1) u)'\,.\label{eq:alteval1}
\end{equation}
To construct the Evans function, we write \eqref{eq:eval} as a first-order system. To do so, we define
$W:=(u,z,u',z')^\tr$, and we see that the eigenvalue equation can be written as a linear system
\beq\label{eq:evans_system}
W'=\mathbb{A}(x;\lambda)W\,,
\eeq
where
\beq
\label{eq:Amatrix}
\mathbb{A}(x;\lambda) =
\begin{bmatrix}
0 & 0 & 1 & 0  \\
0 & 0 & 0 & 1 \\
\lambda + \bar u_x-qk\varphi'(\bar u)\bar z & -qk\varphi(\bar u) & \bar u-1 & 0 \\
D^{-1}k\varphi'(\bar u) \bar z & D^{-1}\lambda + D^{-1}k\varphi(\bar u) & 0 & -D^{-1}
\end{bmatrix}\,.
\eeq

In the case of strong detonations, working with the integrated equations has the advantage of removing the translational zero eigenvalue.  While this is not the case for weak detonations, we find that we obtain tighter energy estimates using integrated coordinates.  Thus, we define $w' := u + qz$ so that \eqref{eq:alteval1} becomes
\begin{equation}
u'' = \lambda w' - qz' - qDz'' - qDz'' + ((\bar u-1)u)' \label{eq:integ_u}
\end{equation}
which can be integrated so that the eigenvalue equation becomes
\begin{subequations}\label{eq:integ_system}
\begin{align}
&u' = \lambda w - qz - qDz' + (\bar u-1)u\,, \\
&w' = u+qz\,, \\
&z'' = D^{-1}(\lambda z - z' + k\varphi'(\bar u)u \bar z + k\varphi(\bar u)z)\,.
\end{align}
\end{subequations}
In matrix form with $X: = (u,w,z,z')^\tr$, equation \eqref{eq:integ_system} takes the form
\begin{equation}
X' = \mathbb{B}(x;\lambda)X \label{eq:integ_mform}
\end{equation}
where
\begin{equation} \label{eq:integ_matrix}
\mathbb{B}(x;\lambda) :=
\begin{bmatrix}
\bar u-1                    & \lambda   & -q                                & -qD     \\
1                           & 0         & q                                 & 0       \\
0                           & 0         & 0                                 & 1       \\
D^{-1}k\varphi'(\bar u)\bar z  & 0         & D^{-1}(\lambda + k\varphi(\bar u))   & -D^{-1} \\
\end{bmatrix}.
\end{equation}

Thus we have written the eigenvalue problem as a linear system of first order ODEs where the coefficient matrix depends on $x$ and the spectral parameter $\lambda$.  We observe that, due to Lemma~\ref{lem:expdecay}, the coefficient matrix $\mathbb{B}$ decays exponentially fast as $x\to\pm\infty$ to a limiting matrix $\mathbb{B}_\spm(\lambda)$. The basic idea of the construction of the Evans function is to look for solutions of \eqref{eq:integ_system} which have the ``correct'' asymptotic behavior, as described by the limiting system $X'=\mathbb{B}_\spm(\lambda)X$. Then, roughly speaking, the Evans function can be thought of as a determinant
\[
E(\lambda)=\det(\mathcal{W}^\sp(x;\lambda),\mathcal{W}^\sm(x;\lambda))|_{x=0}
\]
where $\mathcal{W}^\spm$ are bases for the subspaces of solutions of \eqref{eq:integ_mform} that decay at $\pm\infty$. Evidently, a zero of $E(\lambda)$ indicates a linear dependence between these subspaces. Such a linear dependence is equivalent to the existence an eigenfunction. We omit the details of the construction. For more details about the construction of the Evans function for the Majda model, see \cite{LRTZ_JDE07}. For more general background information about the Evans function, see, e.g.,
the survey article of Sandstede \cite{S_HDS02} and  \cites{AGJ_JRAM90,PW_PTRSL92,GZ_CPAM98}.

\subsection{High-frequency bounds}\label{ssec:hfb}

We note that the integrated equations \eqref{eq:integ_system} can be written as
\begin{subequations}
\label{eq:eval5}
\begin{align}
\lambda w - (1-\bar{u}) w' = q\bar{u} z + q(D-1)z' + w''\,, \label{eigenvalue5:a}\\
\lambda z + k (\varphi(\bar{u})-q\varphi'(\bar u)\bar z)z = z' + k\varphi'(\bar u)\bar z w' + Dz''\,.\label{eigenvalue5:b}
\end{align}
\end{subequations}
We now show by an energy estimate that any unstable eigenvalue of the integrated eigenvalue equations lies in a bounded region of the unstable half plane.  While energy estimates for this system have been established by Humpherys et al.\ \cite{HLZ_majda}, here, using the monotonicity of weak profiles established above in Proposition \ref{prop:mono}, we are able to obtain a modestly improved estimate.

\begin{prop}[High-frequency bounds]\label{prop:hfb}
Any eigenvalue $\lam$ of \eqref{eq:eval5} with nonnegative real part satisfies
\beq\label{eq:bound}
\re\lambda+|\im\lambda|\leq \max\left\{ 3, \frac{1}{4 D} + \left( \frac{1}{4} + \frac{1}{2}|D-1|^2\right) k L + k M \right\}
\eeq
where
\beq\label{eq:mdef}
L:= \sup_{x\in\RR}{\varphi'(\bar u(x))\bar z} \quad\text{and}\quad M:=\sup_{x\in\RR} \left((1+q) \varphi'(\bar{u}) \bar{z}-\varphi(\bar{u})\right).
\eeq
\end{prop}

\begin{proof}
We multiply \eqref{eigenvalue5:a} by $w^*$ and \eqref{eigenvalue5:b} by $z^*$ and integrate (we integrate the $w'' w^*$, $z''z^*$ and $z'w^*$ terms by parts) to give
\begin{subequations}
\label{energy}
\begin{align}
&\lambda \intr |w|^2 + \intr |w'|^2 - \intr (1-\bar{u}) w'w^*= q \intr \bar{u} zw^* - q(D-1)\intr zw^{*'}, \label{energy:a}\\
&\lambda \intr |z|^2 + D \intr |z'|^2 + k \intr (\phi(\bar{u})- q \phi'(\bar{u}) \bar{z}) |z|^2   = \intr z'z^* - k \intr\phi'(\bar{u})\bar{z} w'z^*. \label{energy:b}
\end{align}
\end{subequations}
Taking the real part of \eqref{energy}, we find
\begin{subequations}
\label{eq:real}
\begin{align}
&\re\lambda \intr |w|^2 + \intr |w'|^2 - \re \left(\intr (1-\bar{u}) w'w^*\right) = \re \left(q \intr \bar{u} zw^* - q(D-1)\intr zw^{*'} \right), \label{real:a}\\
&\re\lambda \intr |z|^2 + D \intr |z'|^2 + k \intr (\varphi(\bar{u})- q \varphi'(\bar{u}) \bar{z}) |z|^2 = -\re \left(k \intr\varphi'(\bar{u})\bar{z} w'z^*\right). \label{real:b}
\end{align}
\end{subequations}
Similarly, taking the imaginary part of \eqref{energy}, we observe
\begin{subequations}
\label{eq:imag}
\begin{align}
&\im\lambda \intr |w|^2 - \im\left(\intr (1-\bar{u}) w'w^*\right) = \im\left(q \intr\bar{u} zw^* - q(D-1) \intr zw^{*'}\right)\,, \label{imag:a}\\
&\im\lambda \intr |z|^2 = \im\left(\intr z'z^* - k \intr \phi'(\bar{u})\bar{z} w'z^*\right) = 0\,. \label{imag:b}
\end{align}
\end{subequations}
Noting that
\begin{equation}
\re\left(\intr (1-\bar u)w'w^*\right) = \frac{1}{2}\intr \bar u' |w|^2 < 0 \label{eq:mono_improv}
\end{equation}
and combining \eqref{eq:real} and \eqref{eq:imag}, we see that
\begin{multline}
\label{eq:ri-a}
\big(\re\lambda + |\im\lambda|\big) \intr |w|^2 + \intr |w'|^2 \\
\leq\sqrt{2} q \intr \bar{u} |z||w| + \sqrt{2} q |D-1| \intr |z| |w'| + \intr |1-\bar{u}| |w'| |w|,
\end{multline}
and
\begin{multline}
\label{eq:ri-b}
 \big(\re\lambda + |\im\lambda|\big) \intr |z|^2 + k \intr (\varphi(\bar{u})- q \varphi'(\bar{u}) \bar{z}) |z|^2 + D \intr |z'|^2\\
 \leq \intr |z'| |z| +  \sqrt{2} k \intr |\phi'(\bar{u})\bar{z}| |w'| |z|\,.
\end{multline}
Using Young's inequality (several times) together with the assumption that $\re\lambda\geq 0$, we find that inequalities \eqref{eq:ri-a} and \eqref{eq:ri-b} imply
\begin{multline}
\label{eq:young-a}
\big(\re\lambda + |\im\lambda|\big) \intr |w|^2 + \intr |w'|^2 \leq \sqrt{2} q \norm{\bar{u}}{\infty} \intr \left(\varepsilon_1 |z|^2 + \frac{|w|^2}{4\varepsilon_1}\right) \\
\quad + \sqrt{2} q |D-1| \intr \left(\varepsilon_2 |z|^2 + \frac{|w'|^2}{4\varepsilon_2}\right) + \norm{1 - \bar{u}}{\infty} \intr \left(\varepsilon_3 |w'|^2 +\frac{|w|^2}{4\varepsilon_3}\right)
\end{multline}
and
\begin{multline}
\label{eq:young-b}
\big(\re\lambda + |\im\lambda|\big) \intr |z|^2 + k \intr (\varphi(\bar{u})- q \varphi'(\bar{u}) \bar{z}) |z|^2 + D \intr |z'|^2 \\
\quad \leq \intr \left(\varepsilon_4 |z'|^2 + \frac{|z|^2}{4\varepsilon_4}\right) + \sqrt{2}k L \intr \left(\varepsilon_5 |w'|^2 + \frac{|z|^2}{4\varepsilon_5}\right)\,.
\end{multline}
We multiply \eqref{eq:young-b} by  $\Theta>0$ and add the result to \eqref{eq:young-a}. The result is
\begin{multline}
\label{eq:bigdaddy}
\big(\re\lambda + |\im\lambda|\big)\left( \intr |w|^2 + \Theta |z|^2\right) + k \intr \Phi(x) |z|^2 + \intr |w'|^2 + \Theta D \intr |z'|^2\\
\leq \intr R_1(x)\Theta |z|^2 + \varepsilon_4 \Theta  \intr |z'|^2
+   R_2\intr |w'|^2 +  R_3 \intr |w|^2 \,.
\end{multline}
where
\begin{align*}
\Phi(x)&=(\varphi(\bar{u})- q \varphi'(\bar{u})\bar{z}) \,,\\
R_1&= \frac{\sqrt{2} \varepsilon_1 q \norm{\bar{u}}{\infty}}{\Theta}  + \frac{\sqrt{2} \varepsilon_2 q |D-1|}{\Theta} + \frac{1}{4\varepsilon_4} + \frac{\sqrt{2} k L}{4\varepsilon_5}\,,\\
R_2&=\sqrt{2}\left(\frac{q |D-1|}{4\varepsilon_2} + \frac{\varepsilon_3 \norm{1-\bar{u}}{\infty}}{\sqrt{2}} + \varepsilon_5 \Theta k L \right)\,,\\
\intertext{and}
R_3&=\sqrt{2}\left(\frac{q \norm{\bar{u}}{\infty}}{4\varepsilon_1} + \frac{\norm{1-\bar{u}}{\infty}}{4\sqrt{2}\varepsilon_3}\right)\,.
\end{align*}
%
%
Finally, to simplify \eqref{eq:bigdaddy}, we choose
\begin{align*}
\varepsilon_1 &= \frac{\sqrt{2}}{8} &
\varepsilon_2 &= \sqrt{2} q |D-1| \\
\varepsilon_3 &= \frac{2}{8 \|1-\bar{u}\|_\infty} &
\varepsilon_4 &= D \\
\varepsilon_5 &= \frac{\sqrt{2}}{4} &
\Theta &= (kL)^{-1}\,,
\end{align*}
where $L$ and $M$ are as in \eqref{eq:mdef}.   We also note that $\norm{\bar{u}}{\infty} \leq 2$, $\|1-\bar{u}\|_\infty\leq 1$, and $q\leq 1/2$.  Thus, we have
\[
  \big(\re\lambda + |\im\lambda|\big)\intr (|w|^2 + \Theta |z|^2) \leq 3 \intr |w|^2 +C \intr \Theta |z|^2\,,
\]
where
\[
C:= \left( \frac{1}{4 D} + \left( \frac{1}{4} + \frac{1}{2}|D-1|^2\right) k L + k M \right)\,.
\]
The result follows.
\end{proof}

\begin{remark}
We easily obtain the following crude bounds on $L$ and $M$:
\begin{align*}
L & \leq \sup_{x\in\RR}{\varphi'(\bar u(x))} \leq \varphi'\left(u_\mathrm{ig}+\frac{\Ea}{2}\right) = \frac{4}{\Ea}\varphi\left(u_\mathrm{ig}+\frac{\Ea}{2}\right) \leq \frac{4}{\Ea} \me^{-2} \approx \frac{0.5413}{\Ea}\,,\\
M&\leq\sup_{x\in\RR}{(1+q) \varphi'(\bar{u})} \leq \frac{6}{\Ea} \me^{-2} \approx \frac{0.8120}{\Ea}\,.
\end{align*}
\end{remark}

\subsection{Evans Function}\label{ssec:evans}
As outlined above, the Evans function $\lambda\mapsto E(\lambda)$ acts as a kind of characteristic polynomial for the linear operator $L$;  that is,
\[
E(\lambda_0)=0\;\Leftrightarrow\,\text{$\lambda_0$ is an eigenvalue of $L$}\,.
\]
Unfortunately, it is seldom possible to explicitly evaluate the Evans function. However, it is possible to approximate it numerically \cite{HZ_PD06}. Since $E$ is analytic on the unstable half plane, it is possible to seek zeros by winding number computations. The origins of this approach to stability can be found in the work of Evans and Feroe \cite{EF_MB77}. These ideas have since been used to address the stability of traveling-wave solutions to a number of systems of interest; we mention, e.g., \cites{PSW_PD93,AS_NW95,B_MC01,BDG_PD02}.

Techniques for the numerical approximation of the Evans function have been described in detail elsewhere \cites{B_MC01,HSZ_NM06,HZ_PD06}, so we only outline the essential features of the computation here.
\begin{enumerate}
\item We approximate the profile on a finite computational domain $[-M_\sm,M_\sp]$. The computational values for plus and minus spatial infinity, $M_\spm$, must be chosen with some care. Writing the traveling-wave equation \eqref{eq:tw} as $U'=F(U)$ together with the condition that $U\to U_\spm$ as $\xi\to\pm\infty$, the typical requirement is that $M_\spm$ should be chosen so that
$ |U(\pm M_\spm)-U_\spm|$
is within a prescribed tolerance.
\item For each profile compute the high-frequency spectral bounds.  To do so we must compute $L$ and $M$ from \eqref{eq:mdef}.  With these values in hand, we can determine a positive real number $R$ large enough to guarantee that no eigenvalues of \eqref{eq:integ_system} lie outside $B_R^+$, the half circle of radius $R$ in the positive half-plane $\re\lambda \geq 0$.  Now we need only establish that the Evans function has no zeros in the bounded region $B_R^+$.
\item Given the solution profiles and appropriate bound from the previous step, we evaluate the Evans function by use of the \textsc{StabLab} package, a \textsc{MatLab}-based package developed for Evans function computation \cite{STABLAB}.  We use the polar-coordinate method \cite{HZ_PD06} for the computation and Kato's method \cite{Kato}*{p. 99} to analytically determine the initial eigenvectors;  details of these methods are described in \cites{BZ_MC02,BDG_PD02,HSZ_NM06}.  Throughout our study, we set the tolerances on \textsc{Matlab}'s stiff ODE solver {\tt ode15s} to be {\tt RelTol = 1e-6} and {\tt AbsTol = 1e-8}.
\item We compute the number of zeros of the Evans function inside the contour $S = \partial B^+_R$ by computing the winding number of the image of $S$ under the Evans function.  This is also computed using the \textsc{StabLab} package by choosing a set of $\lambda$-values on $S$ for which we sum the changes in $\arg E(\lambda)$ as we traverse $S$ counterclockwise.  We add $\lambda$ values to our set if the change in $\arg E(\lambda)$ is greater than 0.2 in any step.  By Rouch\'e's theorem, we are guaranteed to have an accurate computation of the winding number if the argument varies by less than $\pi/2$ between two $\lambda$ values \cite{Henrici}.

As mentioned previously, the shift to integrated coordinates does not remove the translational zero eigenvalue for weak detonations.  In order to use the winding number technique just described, we must first remove the zero eigenvalue in another way.  In this case we find that since the zero eigenvalue has multiplicity one, we simply divide the Evans function by $\lambda$ to remove that zero.
\end{enumerate}

\section{Experiments}\label{sec:experiments}
\subsection{Experimental framework and parameter space}
We now discuss our experiments.  Recall that $s = 1$, and $u_\sm \in (u_+,1)$ is determined by \eqref{eq:uminus}.  Also, the parameter $k$ was used in inflating the state space and thus given values for the other parameters, a value of $k$ for which a solution can be found is determined by the ODE solver.  While $u_\sp$ can take values specified by \eqref{eq:uplusq}, we find no qualitative differences for varied values of $u_\sp$ and thus we fix $u_\sp = 0$ for our experiments here.  We also find no qualitative difference in letting $\ui$ vary and so we set $\ui= 0.1$ throughout.  These values of $u_\sp$ and $\ui$ correspond with those used in \cite{HLZ_majda}.  In this case \eqref{eq:uminus} and \eqref{eq:uplusq}, imply $q \in [0,\frac{1}{2})$ in order to ensure $u_\sm < 1$, the condition for a weak detonation.

We let the parameters $q$, $\Ea$, and $D$ vary through the ranges $[0.250,0.499]$, $[10^{-3},6]$, and $[10^{-3},15]$ respectively.  We find all computed profiles to be stable. However, we were not able to obtain solutions for all values of the parameters in the specified region. More precisely, we attempted to compute solutions for all parameter values in the grid
defined by $\mathscr{P}=\mathscr{Q}\times\mathscr{E}_A\times\mathscr{D}$ where 
\begin{multline*}
\mathscr{Q}= \{0.4990, 0.4764, 0.4537, 0.4311, 0.4085, 0.3858, 0.3632, 0.3405\} \\
	\cup\{0.3179, 0.2953, 0.2726, 0.2500\}\,,
\end{multline*}
\begin{multline*}
\mathscr{E}_A=\{10^{-3},  0.1437, 0.2864, 0.4291, 0.5719, 0.7146, 0.8573, 1\} \\
\cup\{1.5556, 2.1111, 2.6667, 3.2222, 3.7778, 4.3333, 4.8889, 5.4444, 6\}\,,
\end{multline*}
and
\begin{multline*}  
\mathscr{D}= \{10^{-3}, 0.1437 , 0.2864, 0.4291, 0.5719, 0.7146, 0.8573, 1, 2.5556, 4.1111\} \\
	\cup\{5.6667, 7.2222, 8.7778, 10.3333, 11.8889, 13.4444, 15\}\,.
\end{multline*}
Our computational grid in parameter space thus contains $12\times17\times17 = 3468$ combinations of parameters. However, we were unable to solve the traveling-wave equation \eqref{eq:tw} within our specified tolerances for a number of points in $\mathscr{P}$.  In particular, as we move toward the boundaries of this parameter space we typically find that we must take smaller continuation steps until we hit a point that continuing is no longer feasible; the successes in various regions of $\mathscr{P}$ are tabulated in Table \ref{tab:success}.
\begin{table}
\begin{centering}
\begin{tabular}{cc|c}
$D$ Range & $\Ea$ Range  & Success \\ \hline 
$\leq 1$ & $\leq 1$ & 707/768 \\
$\geq 1$ & $\leq 1$ & 861/960 \\
$\leq 1$ & $\geq 1$ & 327/960 \\
$\geq 1$ & $\geq 1$ & 278/1200
\end{tabular}
\end{centering}
\caption{Success rates for the profile solver in various regions of $\mathscr{P}$.}
\label{tab:success}
\end{table}
As the table illustrates, the majority of the troubles come from large $\Ea$ and small $q$, both of
which result in increasingly large values of $k$.  Although we cannot
find solutions for much of the grid with $\Ea\geq1$, we can find at least
one solution for all values of $\Ea$ checked, but often only for larger
values of $q$. We find that as we allow $q$ to decrease, the parameter $k$ increases to extreme values.  For example, if we fix $D = 1$ and $\Ea = 1$, then for $q = 0.499$, the corresponding value of $k$ is $8.177$.  However for $q = 0.250$, $k$ increases to approximately $7,913$.

\br[Existence of Profiles]\label{rem:failure}
One of the striking features of Table \ref{tab:success} is the relatively poor performance of the profile solver when $\Ea\geq1$. Indeed, the failure rate is really striking for those points in $\mathscr{P}$ corresponding to largest tested values of $D$. However, this is perhaps not so surprising due to the fact that these larger values of $D$ represent a mismatch in the strength of diffusive effects (``gas-dynamical'' viscosity versus species diffusion) in the model; one might expect the resulting multiple scales to introduce a kind of stiffness into the traveling-wave computation. 

Moreover,  we note that in the limit of large activation energy, $\phi\to0$, and the explicit solution to the traveling-wave equation for $\phi\equiv0$ (see \S\ref{ssec:properties}) fails to satisfy the boundary condition at $-\infty$.  The process of increasing the parameter $\Ea$ flattens and decreases the values of the ignition function in for $u$ values in $(u_\sp,u_\sm^\mathrm{w})$,  and our experiments show that the inflation parameter $k$ increases exponentially as $\Ea$ increases. Sample $k$ values  illustrating this growth are given in \S\ref{ssec:energy} below.  
Evidently, the effect of this growth in $k$ is to offset the decay of values of $\phi$ so that that product
$
k\phi(u)
$---which is accounts for the only appearance of $k$ in the traveling-wave equation \eqref{eq:tw}---is stabilized away from values for which no solution exists. However, our imposed error tolerances do not permit an unfettered growth of $k$ in the numerical computations. One possible way to sidestep this difficulty, as discussed in more detail in \S\ref{sec:discussion} below, would be to try using $q$ as the inflation parameter. In fact,  a further investigation of the of the traveling-wave equation including, in particular, a detailed study of the orientation of the one-dimensional unstable manifold at $(u_\sm^\mathrm{w},0,0)$ in the phase space $\RR^3$ would be quite interesting, but
we leave this issue for future work. 
\er

\subsection{Activation energy}\label{ssec:energy}
First we consider activation energy, $\Ea$, allowing it to take values in $\mathscr{E}_A\subset[10^{-3},6]$.  We find that just as $k$ increases with decreasing values of $q$, it also increases as $\Ea$ increases.  For example, if we fix $q = 0.499$ $D = 1$, then for $\Ea = 10^{-3}$ we find $k = 0.236$.  However for $\Ea = 6$ we have $k \approx 21,600$.
\begin{figure}
\centering
\subfloat[]{
\includegraphics[width=0.4\textwidth]{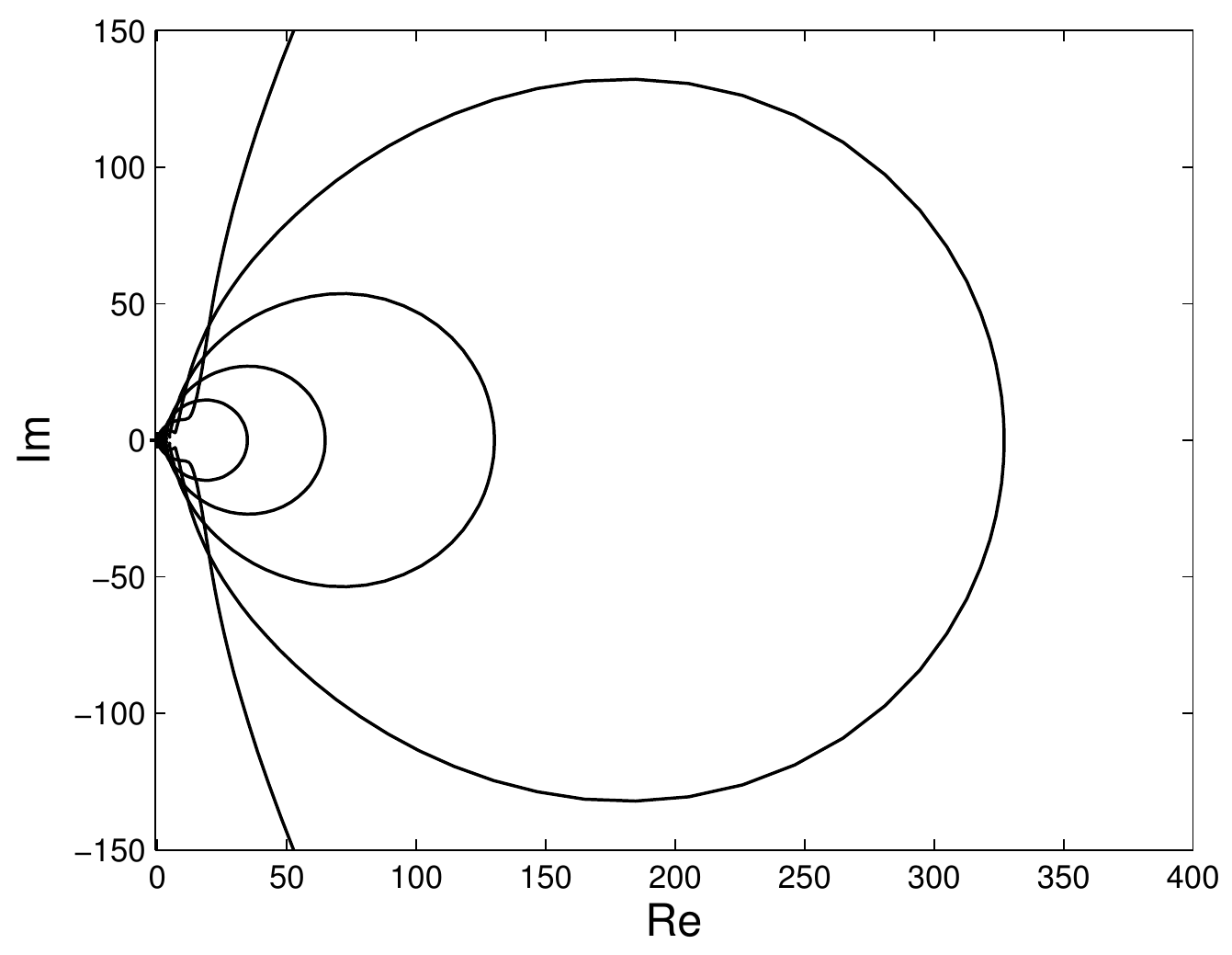}
\label{fig:Large_E}}
\subfloat[]{
\includegraphics[width=0.4\textwidth]{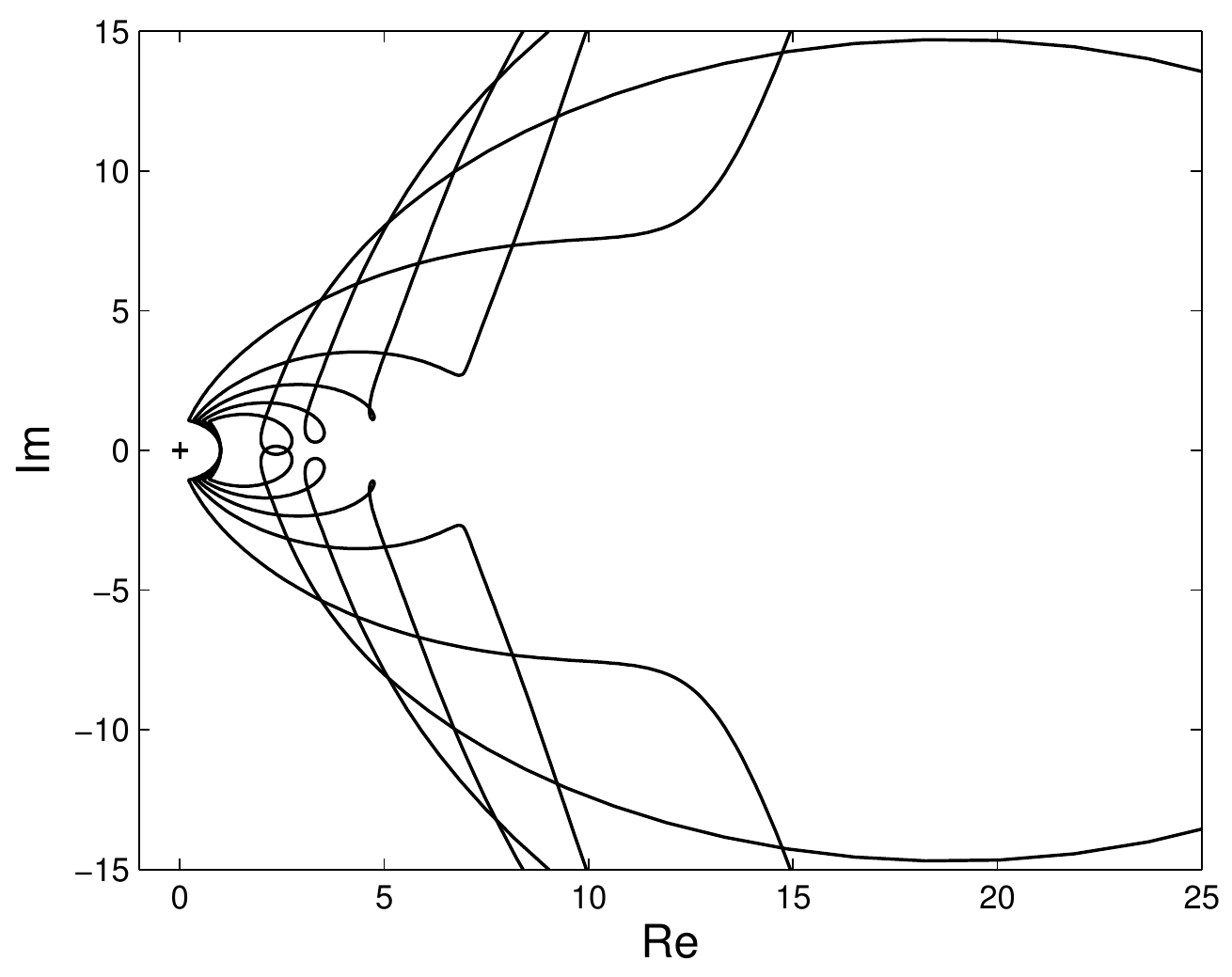}
\label{fig:Large_E_zoom}}\\
\subfloat[]{
\includegraphics[width=0.4\textwidth]{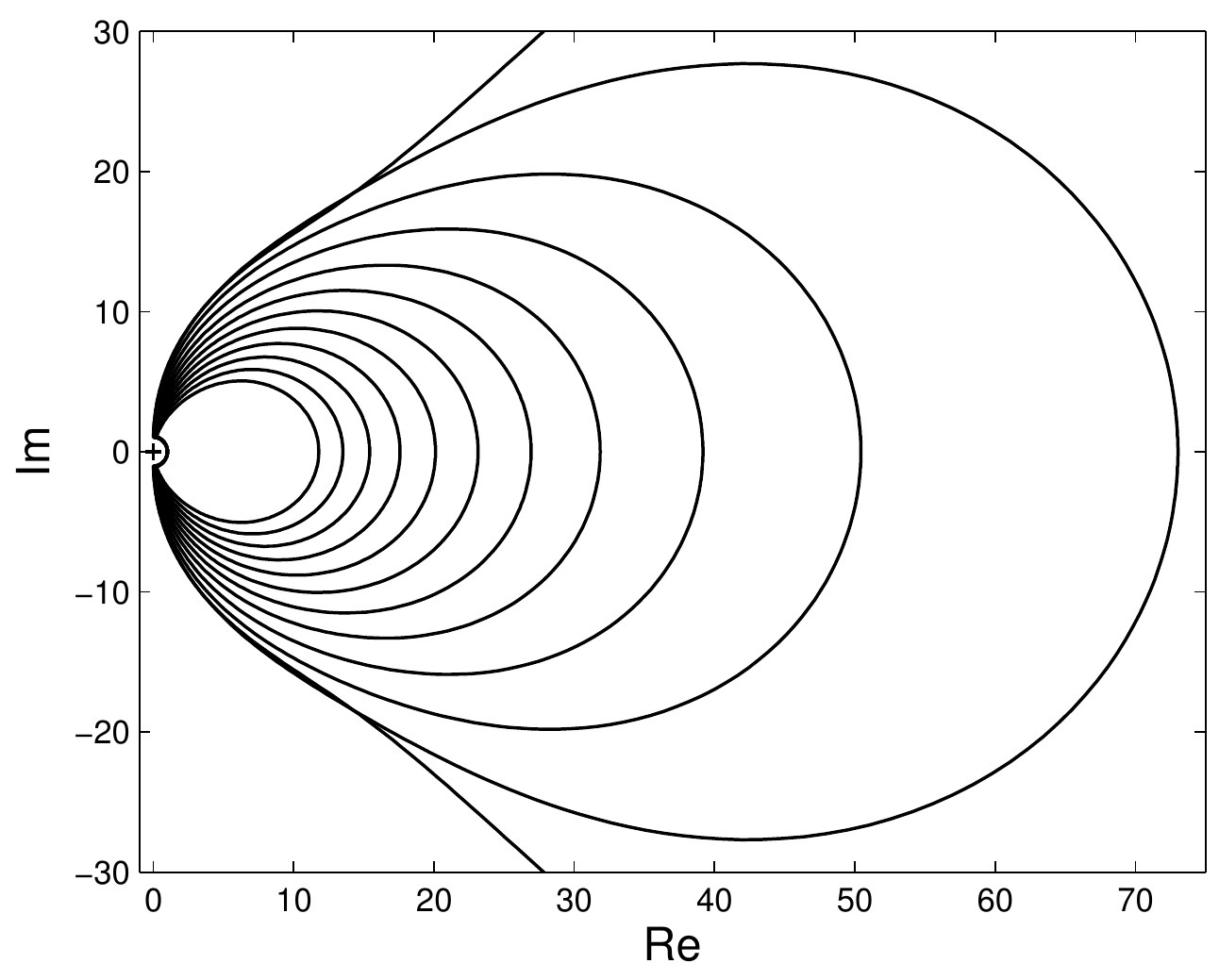}
\label{fig:Small_E}}
\subfloat[]{
\includegraphics[width=0.4\textwidth]{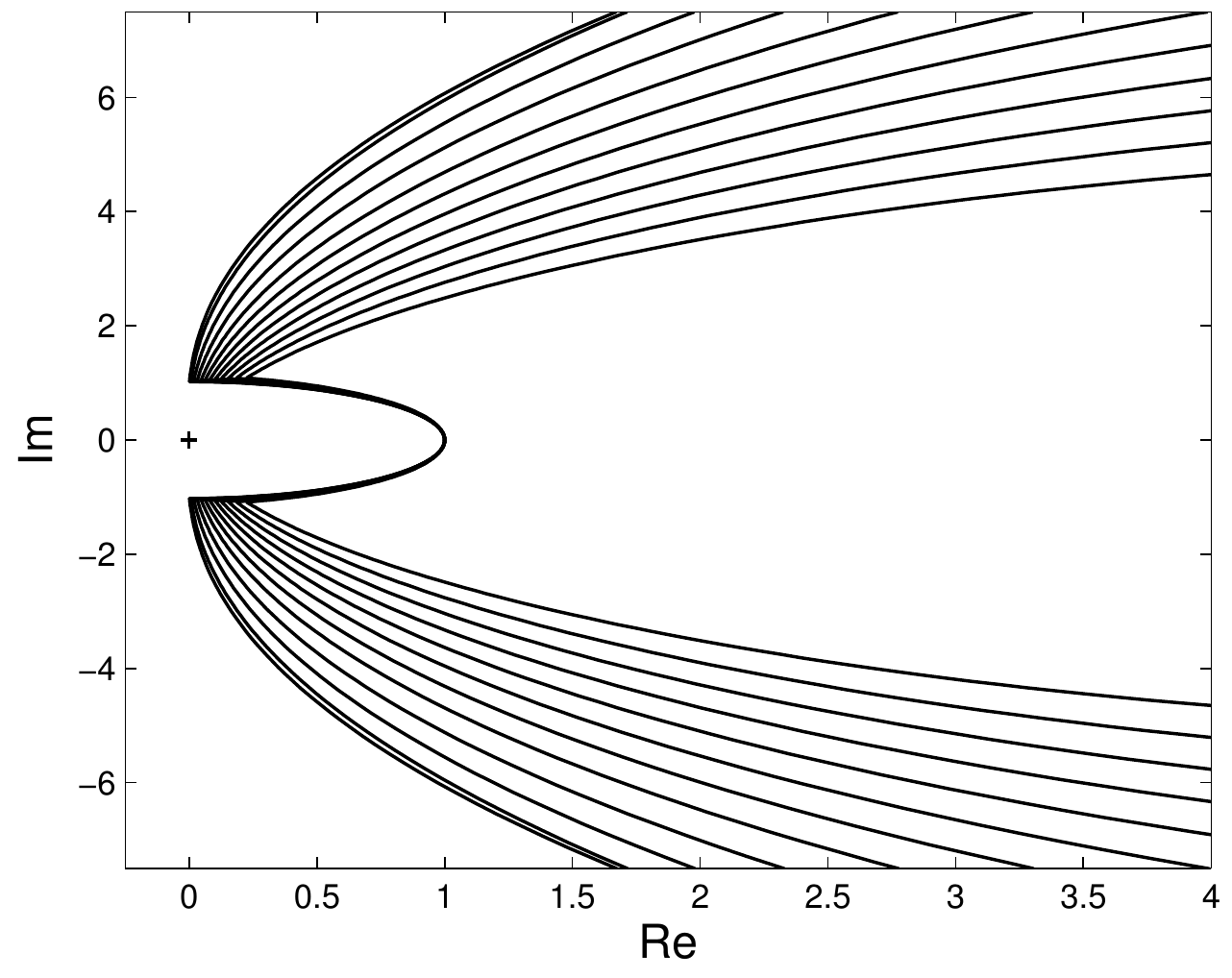}
\label{fig:Small_E_zoom}}
\caption{Evans function output for extreme values of $\mathcal{E}_A$ with $D = 1$.   The parameter values are \protect\subref{fig:Large_E} $\Ea = 3.2$ with $q$ varying through $[.4,.499]$ and \protect\subref{fig:Small_E} $\Ea=10^{-3}$ with $q$ varying through $[.25,.499]$.  Figures \protect\subref{fig:Large_E_zoom} and \protect\subref{fig:Small_E_zoom} are zoomed in versions of \protect\subref{fig:Large_E} and \protect\subref{fig:Small_E} .}
\label{fig:E_evans}
\end{figure}
%
In \textsc{Figure}~\ref{fig:E_evans} \subref{fig:Large_E} and ~\subref{fig:Large_E_zoom}, we see the Evans function output for larger values of $\Ea$. 
%
In \textsc{Figure}~\ref{fig:E_evans} \subref{fig:Small_E} and ~\subref{fig:Small_E_zoom}, we see the Evans function output for small values of  $\Ea$. 

\subsection{Viscosity}
We also consider values of the viscosity constant $D$ in the range $\mathscr{D}\subset[10^{-3},15]$.
\begin{figure}
\centering
\subfloat[]{
\includegraphics[width=0.4\textwidth]{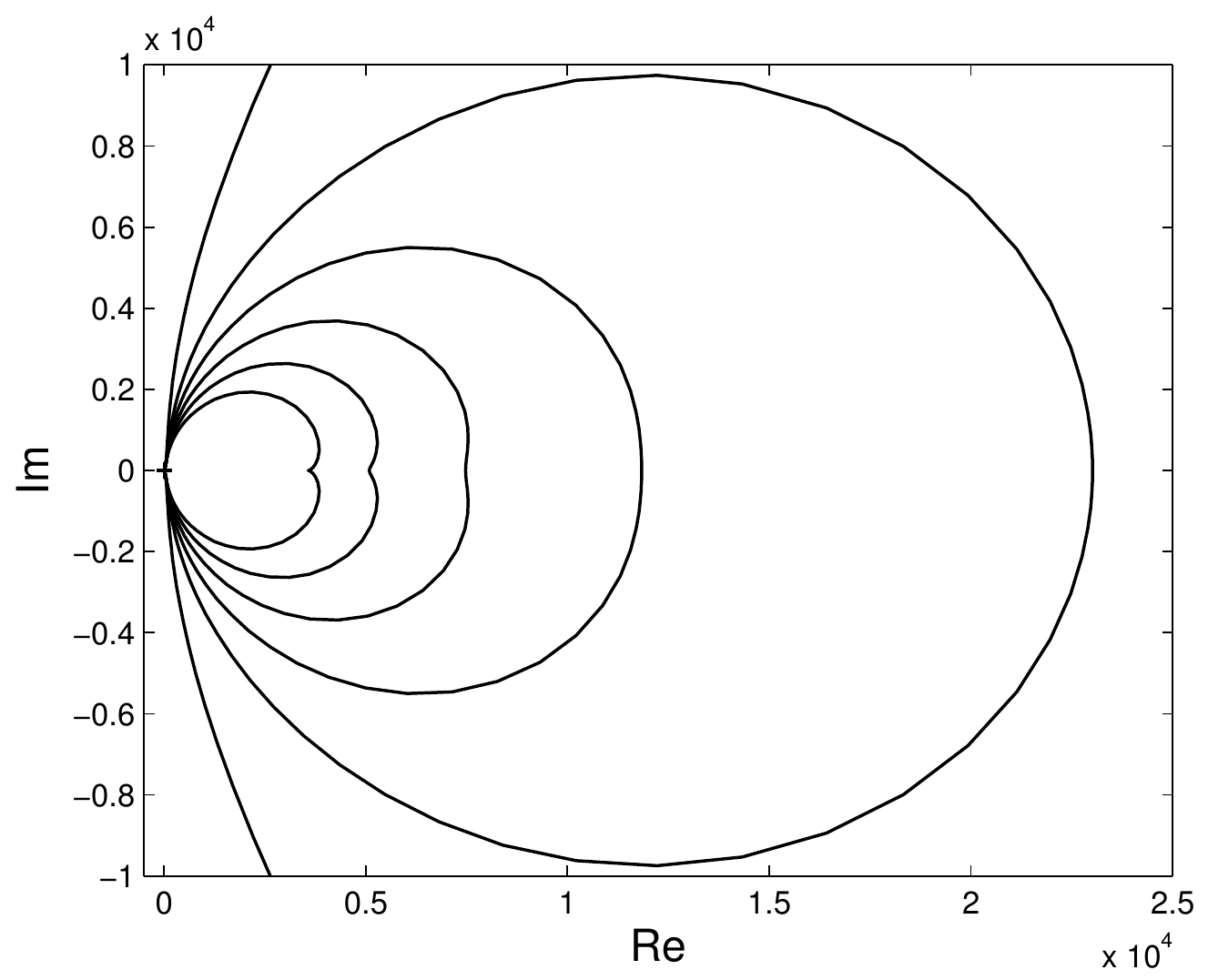}
\label{fig:Large_D}}
\subfloat[]{
\includegraphics[width=0.4\textwidth]{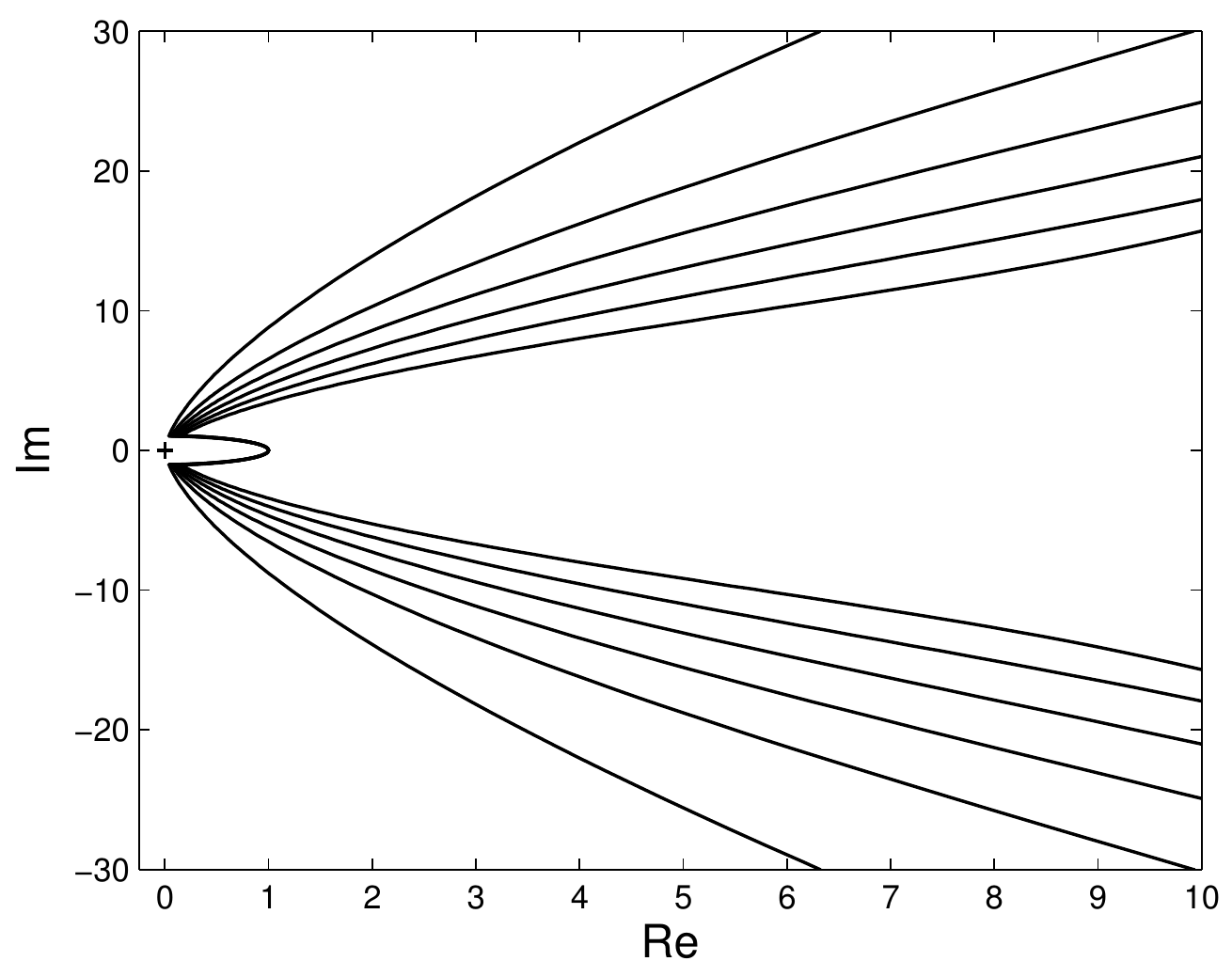}
\label{fig:Large_D_zoom}}\\
\subfloat[]{
\includegraphics[width=0.4\textwidth]{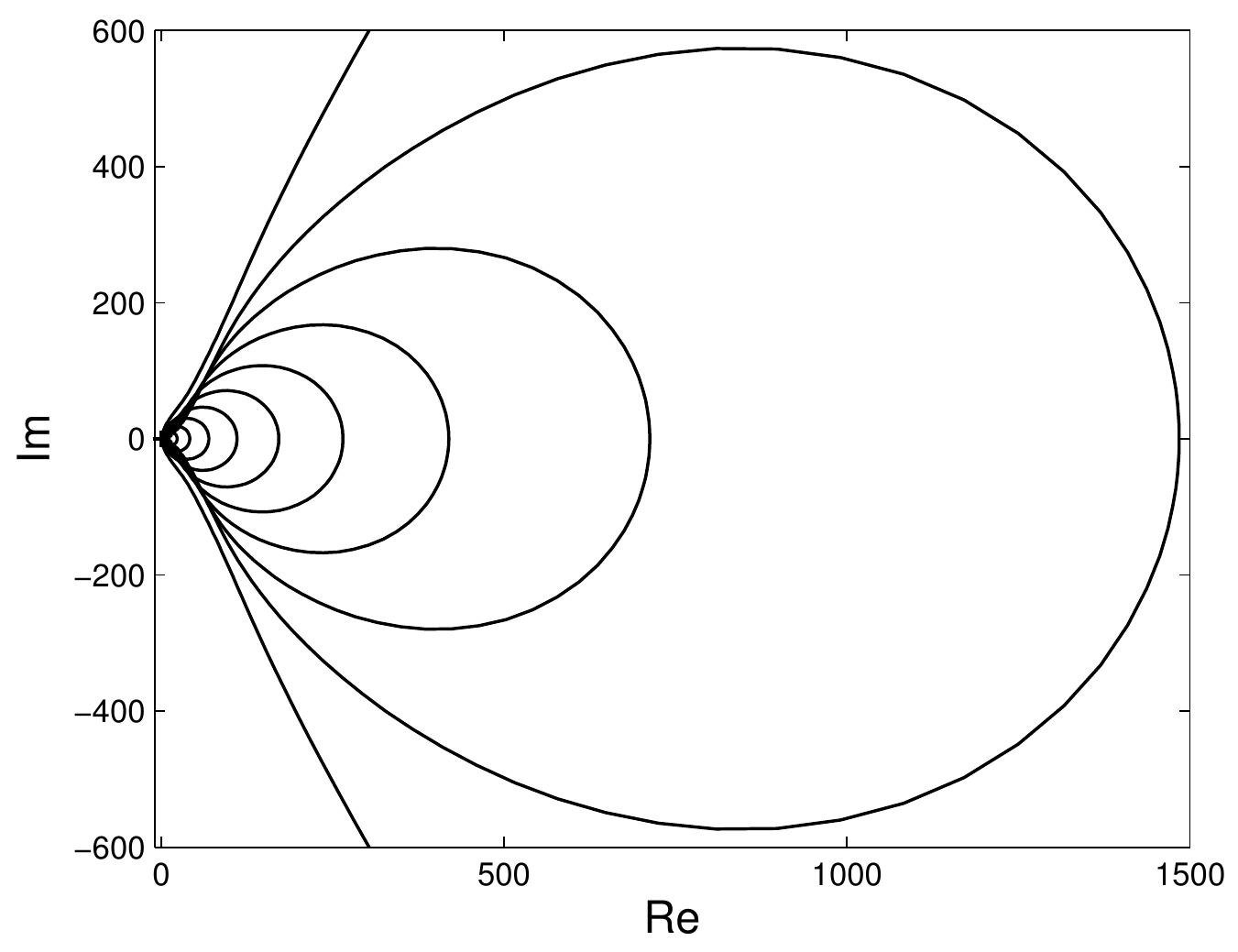}
\label{fig:Small_D}}
\subfloat[]{
\includegraphics[width=0.4\textwidth]{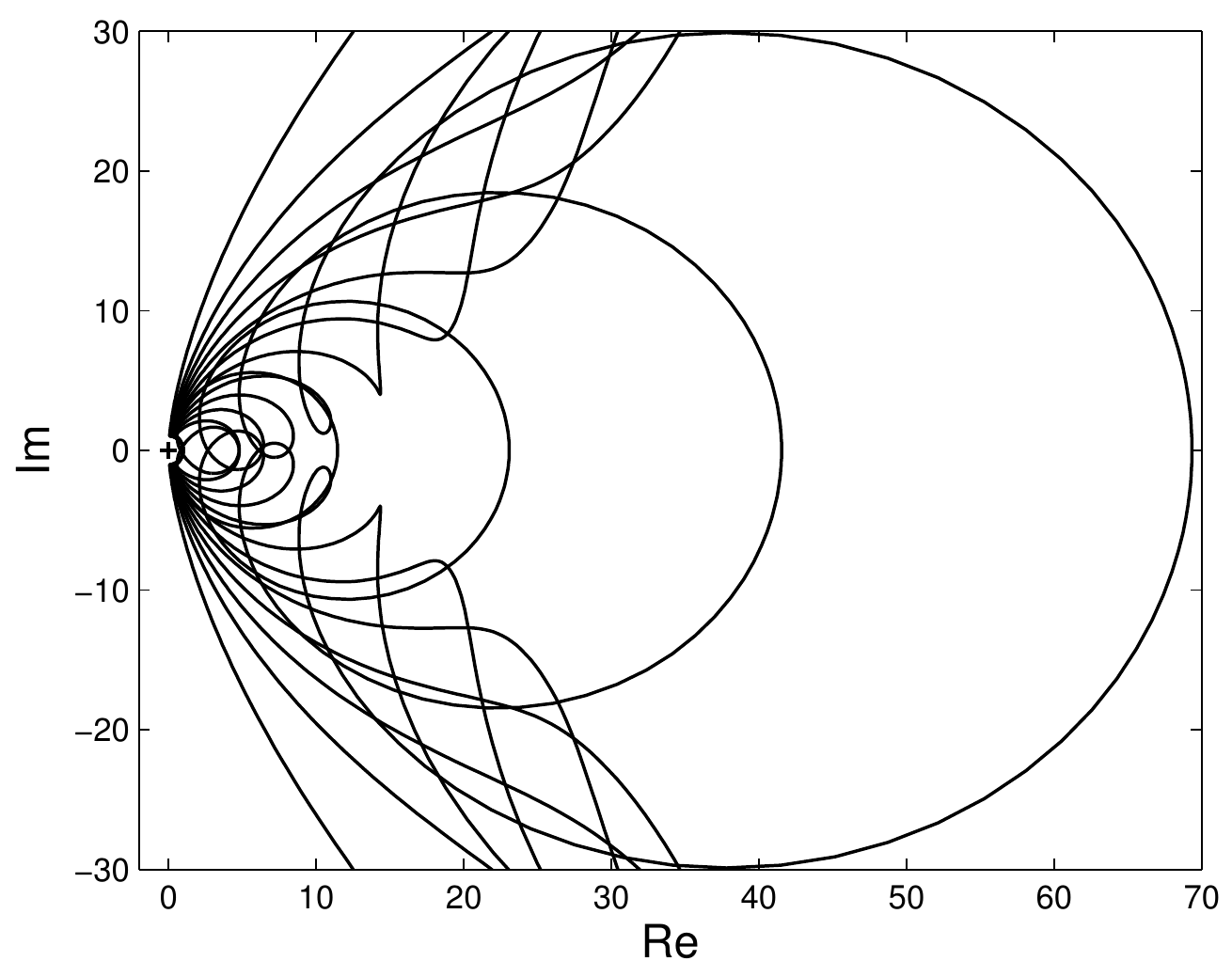}
\label{fig:Small_D_zoom}}
\caption{Evans function output for extreme values of $D$ with $\mathcal{E}_A = 1$.   The parameter values are \protect\subref{fig:Large_D} $D = 15$ with $q$ varying through $[.38,.499]$ and \protect\subref{fig:Small_D} $D=.14$ with $q$ varying through $[.27,.499]$.  Figures \protect\subref{fig:Large_D_zoom} and \protect\subref{fig:Small_D_zoom} are zoomed in versions of \protect\subref{fig:Large_D} and \protect\subref{fig:Small_D} .}
\label{fig:D_evans}
\end{figure}
%
In \textsc{Figure}~\ref{fig:D_evans} \subref{fig:Large_D} and ~\subref{fig:Large_D_zoom}, we see the Evans function output for values of  $D$ in the larger range.
In \textsc{Figure}~\ref{fig:D_evans} \subref{fig:Small_D} and ~\subref{fig:Small_D_zoom}, we see the Evans function output for small values of $D$.

\section{Discussion}\label{sec:discussion}

\subsection{Variations and extensions}
We describe now a handful of variations/extensions of the above analysis that represent interesting possible directions of future research. 
First, we note that the choice of reaction rate $k$ as the inflation parameter in the numerical scheme for approximating solutions of the traveling-wave equation has an unfavorable side effect. Namely, we note that previous Evans-based stability analysis of Majda-type models \cites{HLZ_majda,BZ_majda-znd} have used the freedom in the choice of reaction rate $k$ (related to a choice of spatial scaling)  to keep the half-reaction length inside a reasonable computational domain as $\Ea$ increases. That is, in those previous analyses, $k$ was prescribed as a function of $\Ea$, and larger values of the activation energy were computationally feasible. On the other hand, it should be noted that in those experiments, the high-frequency bounds (analogous to those obtained in Proposition \ref{prop:hfb} above) depend on the value of $k$, and, in the limit of increasing activation energy, the bounds become practically useless. 

In any case, this maneuver is not possible in the framework presented here, and it would be interesting to use $q$ as the inflation parameter and to do a detailed study of those $q$ values in the physical domain $\mathcal{U}$ for which a weak-detonation profile exists. Indeed, as noted above in Remark \ref{rem:failure}, further analysis of the phase space related to the existence of weak detonation profiles may be of interest. 
As noted by Humpherys et al. \cite{HLZ_majda}, the three-dimensional phase space associated with \eqref{eq:tw} has dynamics which are substantially more difficult to describe than those of the planar system associated with Majda's original model \eqref{eq:mm0}. In his original paper, Majda gave a very clear proof establishing the critical value of the heat release parameter $q$ for which weak detonations exist. By contrast, in the physical system, the small-viscosity analysis of Gasser \& Szmolyan \cite{GS_SIAP93} imposes a substantial restriction on the location of the unburned rest state for the existence of a weak detonation profile. This restriction is related to the location of the ignition temperature and, given that the ignition temperature assumption is an artificial mathematical convenience, this also may merit further investigation.
We mention also that other choices for the flux $f$ and the ignition function $\varphi$ might be interesting to analyze. In addition, all of the Evans-based analysis of Majda-type models to date has focused on the simplest possible one-step reaction scheme. It would be quite interesting to repeat the analysis for a model with more realistic chemical kinetics. For example, in the context of the physical system, it is known that in a two-step reaction in which the second reaction is endothermic, new phenomena appear \cite{FickettDavis}.


%

\subsection{Concluding remarks}
As described above, weak detonations are undercompressive, and the stability of such waves has not been treated by weighted-norm techniques. We thus view these results as an important case study illustrating of the flexibility and power of a family of mathematical techniques based on the Evans function. In particular, these results taken together with the pointwise Green function analysis of Lyng et al.\ \cite{LRTZ_JDE07} and the spectral analysis of Humpherys et al.\ \cite{HLZ_majda} give a clear demonstration that both weak (undercompressive) and strong (Lax-type) viscous detonation waves can be analyzed in a common framework.

Taking these results---together with the related analyses of Barker \& Zumbrun \cite{BZ_majda-znd} for the Majda-ZND model (an inviscid Majda model), and of Humpherys et al.\ \cite{HLZ_majda} (viscous strong detonations)---we find a recurring theme of \emph{stability} for detonation-wave solutions of these simplified models for reacting flow. Given the well known instabilities that are known to be present in the physical system, these results suggest that these simplified, scalar models have only limited utility as stand-ins for the physical system.

Perhaps some enhancements or extensions of the model, e.g., a different form of the flux and ignition functions or the inclusion of more realistic chemistry might lead to a model whose detonation-wave solutions have behavior which is more faithful to that of the physical system. In this vein, we note that Radulescu \& Tang \cite{RT_PRL11} have had some success in this endeavor by adding a forcing term to Fickett's model \cite{F_AJP79} which is, more or less, an inviscid Majda model. On the other hand, in ongoing work, several of the authors have extended the Evans-function approach used here to the physical system (the Navier--Stokes equations with chemical reactions); see \cite{rns}. The preliminary results are interesting, and they suggest, in contrast to the results reported above, that the inclusion of diffusive effects has an important effect on the stability characteristics of strong detonation waves. Thus, in addition to treating weak and strong detonations on an equal footing, the Evans techniques that form the basis for this paper also show promise for treating the physical system of equations, and this adds considerably to their mathematical interest.

%


\end{document}